\documentclass[a4paper,11pt]{amsart}
\usepackage[utf8]{inputenc}
\usepackage[stretch=10]{microtype}
\usepackage{mathpazo}

\usepackage{hyphenat} %
\usepackage[all,cmtip,2cell]{xy}
\UseTwocells
\xyoption{2cell}
\usepackage{amsfonts}
\usepackage{amsthm}
\usepackage{amsmath}
\usepackage{amssymb}
\usepackage[mathscr]{euscript}
\usepackage{enumerate}
\usepackage{verbatim} %
\usepackage{mathtools}
\usepackage{url}
\usepackage{tikz-cd}
\usetikzlibrary{decorations.pathmorphing,shapes}
\usepackage{todonotes}
\usepackage{quiver}
\usepackage{geometry}
\usepackage[pdftex]{hyperref} %

\usepackage[nameinlink,capitalise,noabbrev]{cleveref}
\theoremstyle{definition} \newtheorem{defn}[equation]{Definition}
\theoremstyle{definition} \newtheorem{prop-def}[equation]{Proposition-Definition}
\theoremstyle{plain} \newtheorem{thm}[equation]{Theorem}
\theoremstyle{plain} \newtheorem{lemma}[equation]{Lemma}
\theoremstyle{plain} \newtheorem{prop}[equation]{Proposition}
\theoremstyle{plain} \newtheorem{cor}[equation]{Corollary}
\theoremstyle{plain} \newtheorem*{thm*}{Theorem}
\theoremstyle{remark} \newtheorem{rmk}[equation]{Remark}
\theoremstyle{definition} \newtheorem{notn}[equation]{Notation}
\theoremstyle{remark} \newtheorem{ex}[equation]{Example}
\numberwithin{equation}{section}

\newtheoremstyle{TheoremNum}
{}{}              %
{\itshape}                      %
{}                              %
{\bfseries}                     %
{.}                             %
{ }                             %
{\thmname{#1}\thmnote{ \bfseries #3}}%
\theoremstyle{TheoremNum}
\newtheorem{thmn}{Theorem}

\newcommand{\F}{\mathbb{F}}

\newcommand{\D}{\mathscr{D}}
\newcommand{\E}{\mathscr{E}}
\newcommand{\M}{\mathcal{M}}
\renewcommand{\L}{\mathcal{L}}
\newcommand{\C}{\mathscr{C}}
\newcommand{\G}{\mathbb{G}}
\newcommand{\Z}{\mathbb{Z}}
\newcommand{\Q}{\mathbb{Q}}
\newcommand{\N}{\mathbb{N}}
\renewcommand{\1}{\mathbf{1}}
\renewcommand{\S}{\mathscr{S}}
\renewcommand{\P}{\mathscr{P}}
\newcommand{\R}{\mathscr{R}}
\newcommand{\bS}{\mathbb{S}}
\newcommand{\bP}{\mathbb{P}}
\newcommand{\id}{\mathrm{id}}

\newcommand{\colim}{\mathrm{colim}}
\newcommand{\cof}{\mathrm{cofib}}
\newcommand{\op}{\mathrm{op}}
\newcommand{\Map}{\mathrm{Map}}

\newcommand{\adj}[4]{\xymatrix{#1 \ar@<.2pc>[r]^-{#3} & \ar@<.2pc>[l]^-{#4} #2 }}
\newcommand{\eqv}[4]{\xymatrix{#1 \ar@<.5pc>[r]^-{#3} & \ar@<.5pc>[l]^-{#4}_-\sim #2 }}
\newcommand{\Fin}{\mathrm{Fin}}
\newcommand{\Lan}{\mathrm{Lan}}
\newcommand{\CAlg}{\mathrm{CAlg}}
\newcommand{\Sp}{\mathrm{Sp}}
\newcommand{\SP}{\mathrm{SP}^\infty}
\newcommand{\Spc}{\mathrm{Sp}^{\textup{cn}}}
\newcommand{\Pic}{\mathrm{Pic}}

\newcommand{\Mod}{\mathrm{Mod}}

\newcommand{\Fun}{\mathrm{Fun}}
\newcommand{\FunL}{\mathrm{Fun}^\mathrm{L}}
\newcommand{\FunR}{\mathrm{Fun}^\mathrm{R}}

\newcommand{\sma}{\wedge}
\newcommand{\res}{\mathrm{res}}
\newcommand{\ev}{\mathrm{ev}}
\newcommand{\cy}{\mathrm{cy}}
\newcommand{\ER}{E_\infty\mbox{-}R}

\newcommand{\End}{\mathscr{E}\textup{nd}}

\renewcommand{\subset}{\subseteq}

\newcommand{\prl}{\mathscr{P}\mathrm{r}^\mathrm{L}}
\newcommand{\prlpt}{\mathscr{P}\mathrm{r}^\mathrm{L}_{\mathrm{Pt}}}
\newcommand{\prlsem}{\mathscr{P}\mathrm{r}^\mathrm{L}_{\mathrm{Sem}}}
\newcommand{\prladd}{\mathscr{P}\mathrm{r}^\mathrm{L}_{\mathrm{Add}}}
\newcommand{\prlst}{\mathscr{P}\mathrm{r}^\mathrm{L}_{\mathrm{St}}}

\newcommand{\Mon}{\mathrm{Mon}_{E_\infty}}
\newcommand{\Grp}{\mathrm{Grp}_{E_\infty}}

\newcommand{\NatCol}{\mathscr{N}\mathrm{at}\mathscr{C}\mathrm{ol}}

\newcommand{\Nat}{\mathscr{N}\mathrm{at}}
\newcommand{\Kan}{\textup{Kan}}

\newcommand{\comma}{,}
\DeclareMathAlphabet{\mathbbe}{U}{bbold}{m}{n}
\def\DDelta{{\Delta}}

\newcommand{\CP}{{\C_{/P}}}
\newcommand{\CPP}{\C_{/P\oplus P}}
\newcommand{\tplus}{\widetilde{+}}

\definecolor{bubblegum}{rgb}{0.99, 0.76, 0.8}

\DeclareUnicodeCharacter{00A0}{ } %

\definecolor{dark-red}{rgb}{0.6,.15,.15}
\definecolor{dark-blue}{rgb}{.1,.1,.65}
\definecolor{dark-green}{rgb}{.1,.65,.1}
\hypersetup{pdfauthor={Nima Rasekh, Bruno Stonek, Gabriel Valenzuela},colorlinks=true, citecolor=dark-blue, urlcolor=dark-green, linkcolor=dark-red,breaklinks=true,hypertexnames=false}
\title{Thom spectra, higher $THH$ and tensors in $\infty$-categories}
\author{Nima Rasekh}
\address{{\'E}cole Polytechnique F{\'e}d{\'e}rale de Lausanne, SV BMI UPHESS, Station 8, CH-1015 Lausanne, Switzerland}
\email{nima.rasekh@epfl.ch}
\author{Bruno Stonek}
\address{Institute of Mathematics, Polish Academy of Sciences,  ul. {\' S}niadeckich 8, 00-656 Warszawa, Poland}
\email{bruno@stonek.com}
\author{Gabriel Valenzuela}
\address{Max Planck Institute for Mathematics, Vivatsgasse 7, 53111 Bonn, Germany}
\email{gvalenzuelarg@gmail.com}

\begin{document}
\date{\today}

\begin{abstract}
Let $f:G\to \Pic(R)$ be a map of $E_\infty$-groups, where $\Pic(R)$ denotes the Picard space of an $E_\infty$-ring spectrum $R$. We determine the tensor $X\otimes_R Mf$ of the Thom $E_\infty$-$R$-algebra $Mf$ with a space $X$; when $X$ is the circle, the tensor with $X$ is topological Hochschild homology over $R$. We use the theory of localizations of $\infty$-categories as a technical tool: we contribute to this theory an $\infty$-categorical analogue of Day's reflection theorem about closed symmetric monoidal structures on localizations, and we prove that for a smashing localization $L$ of the $\infty$-category of presentable $\infty$-categories, the free $L$-local presentable $\infty$-category on a small simplicial set $K$ is given by presheaves on $K$ valued on the $L$-localization of the $\infty$-category of spaces. 

If $X$ is a pointed space, a map $g: A\to B$ of $E_\infty$-ring spectra \emph{satisfies $X$-base change} if $X\otimes B$ is the pushout of $A\to X\otimes A$ along $g$. Building on a result of Mathew, we prove that if $g$ is étale then it satisfies $X$-base change provided $X$ is connected. We also prove that $g$ satisfies $X$-base change provided the multiplication map of $B$ is an equivalence. Finally, we prove that, under some hypotheses, the Thom isomorphism of Mahowald cannot be an instance of $S^0$-base change.
\end{abstract}

\maketitle

\section{Introduction}

Topological Hochschild homology ($THH$) of $E_\infty$-ring spectra $R$ is equivalent to $S^1\otimes R$, the tensor of $R$ with the circle $S^1$ \cite{msv}. %
Tensors of $E_\infty$-ring spectra with other spaces also give interesting invariants. For example, consider tensoring with $S^n$ for $n\geq 2$. For ordinary rings, this was first considered by Pirashvili \cite{pirashvili}, who called it ``higher order Hochschild homology'' and used it in relation to Hodge decompositions. One can also consider tensoring with tori $T^n$. In this case, the action of $T^n$ on $T^n\otimes R$ leads to a higher version of topological cyclic homology. In analogy to the $n=1$ case, it is expected to bear a connection to $n$-fold iterated algebraic $K$-theory. See \cite{bcd}, \cite{cdd}.

Sometimes, it is actually easier not to focus on the specific case of $THH(R)$ but instead look at the more general tensors $X\otimes R$. In this paper, we provide an example of this strategy. We will describe the tensors of a space $X$ with Thom $E_\infty$-ring spectra $Mf$. Examples of the latter include various versions of cobordism spectra like complex cobordism $MU$ or periodic complex cobordism $MUP$. %
More precisely, we will prove:

\begin{thmn}[\ref{thm-tensorthom}]
Let $R$ be an $E_\infty$-ring spectrum. Let $G$ be an $E_\infty$-group and $f:G\to \Pic(R)$ be an $E_\infty$-map. Let $X$ be a pointed space. There is an equivalence of $E_\infty$-$R$-algebras
\[X\otimes_R Mf \simeq Mf \sma \bS[X\odot G].\]
\end{thmn}

Here $\Pic$ denotes the Picard $E_\infty$-group of $R$, $\sma$ denotes the smash product of spectra, $\otimes_R$ denotes the tensor of $E_\infty$-$R$-algebras over spaces, $\odot$ denotes the tensor of $E_\infty$-groups (also known as grouplike $E_\infty$-spaces) over pointed spaces, and $\bS[X\odot G]$ denotes the suspension spectrum of $(X\odot G)_+$ considered as an $E_\infty$-ring spectrum. This reduces the calculation of the homotopy groups of $X\otimes_R Mf$ to that of the unreduced $Mf$-homology groups of $X\odot G$, which is often simpler. For example, in \cite{bcs} the authors determined $THH(Mf)$ as a spectrum, and used Hopkins--Mahowald's interpretation of the Eilenberg--Mac Lane spectra $H\Z$ and $H\F_p$ as Thom spectra to give a new computation of their topological Hochschild homology groups, originally computed by Bökstedt in his foundational manuscript \cite{bokstedt}.

As a consequence of the result above, we get a generalization of the Thom isomorphism theorem of Mahowald \cite{mahowald-thom}, which would here take the form $Mf\sma_R Mf \simeq Mf \sma \bS[G]$. Indeed, that result is obtained by setting $X=*$ in the equivalence \[Mf \sma_R (X\otimes_R Mf) \simeq Mf \sma \bS[X\otimes G]\]
of \cref{ex-thom-iso}. Here $\otimes$ denotes the tensor of $E_\infty$-groups over spaces.

The proof of the theorem is divided in two parts. First, one proves that tensoring with a space and taking Thom spectrum are operations that commute in an adequate sense (\cref{prop-tensorcommutes}). Using an $\infty$-categorical version of the splitting lemma for short exact sequences in abelian categories, we obtain a splitting of $E_\infty$-groups \[X\otimes G \simeq G \times (X\odot G)\] (\cref{cor-tensorsplit}) which we combine with the monoidality of the Thom spectrum construction to finish the proof. When one takes $X=S^1$, the above splitting becomes the well-known splitting of the cyclic bar construction of $G$ as a product of $G$ and the bar construction of $G$. Thus, our proof of the theorem relies on structural properties satisfied by the Thom construction and by tensors in $\infty$-categories, and on a splitting result for tensors of $E_\infty$-groups.\\

Before saying a word about these tensors, we would like to note the differences between the theorem above and the main result of \cite{schlichtkrull-higher}, which is similar. On one hand, a version of the tensor which allows for coefficients in an $Mf$-module is considered by Schlichtkrull. On the other hand, he proves his results for maps $f:G\to BGL_1(\bS)$. Our result is more general in two different ways: first, we consider the whole Picard space instead of only $BGL_1$. This is already an interesting extension, since it allows for non-connective Thom spectra such as  $MUP$ (see \cref{ex-mup}). Second, we allow the Picard space of any $E_\infty$-ring spectrum as a codomain for $f$, instead of only the one of the sphere spectrum. See \cref{rmk-before}, where we also recall the related result of \cite{klang-thom} on factorization homology of Thom spectra.

Note that, at the beginning of Section 4 of \cite{schlichtkrull-higher}, the author sketches a proof of his result, but then notes that ``when trying to make this argument precise, one encounters several technical difficulties'', which he explains. Those technical difficulties are model-categorical in nature, and the author works around them model-categorically as well, using for example different models for $E_\infty$-monoids and introducing a model of the tensor which is homotopy invariant but combinatorially involved. We claim that these complications are mostly side-effects of the rigidity of the model-categorical framework, and in particular of the rigidity of the model for Thom spectra. Our determination of $X\otimes_R Mf$ uses the $\infty$-categorical Thom spectra machinery introduced in \cite{abghr-infty} and further developed in \cite{abg}, as well as the universal property for Thom ring spectra of \cite{barthel-antolin}.

Note as well that Schlichtkrull's theorem features, in lieu of what we have denoted $X\odot G$, the infinite loop space $\Omega^\infty(B^\infty G \sma \Sigma^\infty X)$. We will prove in \cref{prop-tensor_grouplike} that the two constructions coincide: it will be a direct consequence of the formal properties of tensors. In \cref{rmk-kuhn-sp}, we also prove that $X\odot G$ is equivalent to the infinite symmetric product $\SP(X,G)$, a construction studied in \cite{kuhn}.\\

The presence of different tensors like $X\otimes_R Mf$ and $X\odot G$ interacting with each other makes it important to place them on a firm technical ground. Consequently, the first two sections are devoted to their study. The technical foundation for our paper will be that of $\infty$-categories \cite{htt}. As noted by Lurie \cite[4.8]{ha}, the $\infty$-category of presentable $\infty$-categories and left adjoint functors, $\prl$, has a symmetric monoidal product $\otimes$. If $\R$ is a commutative algebra in $\prl$ and $\C$ is a module over it, we say that $\C$ is \emph{tensored} over $\R$. This generalizes the notion of a category being enriched, tensored and cotensored over a symmetric monoidal category to an $\infty$-categorical setting in a succinct way, at least in the presentable case.

As noted by Lurie, the $\infty$-categories of pointed presentable $\infty$-categories and of stable presentable $\infty$-categories are (reflective) localizations of $\prl$, and they are smashing: they are given by $-\otimes \S_*$ and by $-\otimes \Sp$, respectively, where $\S_*$ denotes pointed spaces and $\Sp$ denotes spectra. In \cite{ggn}, the authors noted that both semiadditive (which they call preadditive) and additive presentable $\infty$-categories are similarly smashing localizations of $\prl$, given by $-\otimes \Mon(\S)$ and $-\otimes \Grp(\S)$ respectively; here $\Mon(\S)$ are $E_\infty$-monoids (also known as $E_\infty$-spaces or special $\Gamma$-spaces) and $\Grp(\S)$ are $E_\infty$-groups (very special $\Gamma$-spaces). We will use their developments as the grounding needed for the results presented above, and we will get mileage out of the realization that, if $L,L'$ are two smashing localizations of $\prl$ such that $L'\prl\subseteq L\prl$, then any $L'$-local $\infty$-category is not only tensored over $L'\S$, but also over $L\S$ by restriction of scalars along the map $L\S\to L'\S$ (\cref{prop-res}).

Along the way, we will prove some other $\infty$-categorical results of independent interest. For example, in \cref{prop-dayreflection}, we will give an $\infty$-categorical version of Day's reflection theorem \cite{day-reflection}, which gives equivalent conditions guaranteeing that the localization of a closed symmetric monoidal $\infty$-category is closed symmetric monoidal. Also, we prove that if $L$ is a smashing localization of $\prl$, then $L\S$ is freely generated in $L\prl$ by the monoidal unit of the corresponding symmetric monoidal structure of $L\S$ (\cref{prop-freely-generated}). %
More generally, we prove:

\begin{thmn}[\ref{prop-freelsmod}] Let $K$ be a small simplicial set. Let $L$ be a smashing localization of $\prl$. Then the $\infty$-category of $L\S$-valued presheaves on $K$, denoted by $\P_{L\S}(K)$, is freely generated in $L\prl$ by $K$. More precisely, composition with $v_*\circ j:K\to \P_{L\S}(K)$ induces an equivalence of functor $\infty$-categories
\[\xymatrix@C+1pc{\FunL(\P_{L\S}(K),\C) \ar[r]^-\sim_-{(v_*\circ j)^*} & \Fun(K,\C) } \]
for any $\C\in L\prl$, where $v:\S\to L\S$ is the localization map and $j:K\to \Fun(K^\op,\S)$ is the Yoneda embedding. Here $\Fun^L$ denotes colimit-preserving functors.
\end{thmn}

Having determined $X\otimes Mf$ for a Thom $E_\infty$-ring spectrum $Mf$, we would like to extend this computation to other types of $E_\infty$-ring spectra $B$: for example, to those which admit a map $Mf\to B$. Thus, in the last two sections, we turn to the following more general question. Let $A\to B$ be a morphism of $E_\infty$-ring spectra. For any space $X$, there is an induced map $X\otimes A \to X\otimes B$, and if $X$ is pointed, we can form the square of $E_\infty$-ring spectra
\[\xymatrix{A \ar[r] \ar[d] & B \ar[d] \\ X\otimes A \ar[r] & X\otimes B.}\]
When is this a pushout square? When it is, we say that $A\to B$ \emph{satisfies $X$-base change}. Working over $\Z$, when $X=S^1$ and $A$ and $B$ are ordinary commutative rings, the question amounts to asking when is the natural map $HH(A) \otimes_A B \to HH(B)$ an isomorphism, where $HH$ denotes Hochschild homology. In \cite{weibel-geller}, this is proven to hold when $A\to B$ is étale. This result was generalized to étale extensions of $E_\infty$-ring spectra and topological Hochschild homology in \cite{mathew-thh}. %

Note that a map $g:A\to B$ of $E_\infty$-ring spectra satisfies $S^0$-base change if and only if $g\sma \id:A\sma B \to B \sma B$ is an equivalence. On the other hand, 
the Thom isomorphism theorem of Mahowald, mentioned above, takes the form $Mf\sma Mf \simeq \bS[G] \sma Mf$. One could wonder if this equivalence is induced by a map $g:Mf \to \bS[G]$ that satisfies $S^0$-base change. In \cref{section-thomisonots0} we will prove that this is hardly ever the case, under some reasonable hypotheses on $g$.\\

One can study the question of $X$-base change more generally for a map in a presentable $\infty$-category: the definition is analogous. We will prove the following:

\begin{thmn}[\ref{thm-s1enough}] Let $f:c\to d$ be a map in a presentable $\infty$-category. Let $n\geq 0$. Suppose $f$ satisfies $S^n$-base change. Then $f$ satisfies $X$-base change for any $(n-1)$-connected pointed space $X$.
\end{thmn}
In particular, we deduce that étale extensions of $E_\infty$-ring spectra satisfy $X$-base change when $X$ is connected (\cref{thm-etale-change}), and, following Mathew who proved it for $X=S^1$, we give a condition on algebraic $K$-theory which guarantees that a faithful $G$-Galois extension $A\to B$ satisfies $X$-base change for all connected $X$ (\cref{cor-galois-change}). In general, if $A\to B$ is a faithful $G$-Galois extension, then the question of $X$-base change for a pointed space $X$ is equivalent to $X\otimes A \to X\otimes B$ being a faithful $G$-Galois extension, and to $X\otimes A \to (X\otimes B)^{hG}$ being an equivalence (\cref{prop-x-galois}).

When the multiplication map $B\sma_A B\to B$ of an $E_\infty$-$A$-algebra $B$ is an equivalence, we say that $B$ is \emph{solid}. The units $A\to B$ for these types of algebras give another class of maps which satisfy $X$-base change:

\begin{thmn}[\ref{prop-solid}] Let $A$ be an $E_\infty$-ring spectrum and $B$ be a solid $E_\infty$-$A$-algebra. Then the unit $A\to B$ satisfies $X$-base change for any connected pointed space $X$.
\end{thmn}
For example, the inversion by a homotopy element $x\in \pi_*(R)$ in an $E_\infty$-ring spectrum, $R\to R[x^{-1}]$, satisfies $X$-base change for all connected $X$. In particular, in \cref{cor-tensor-localization} we get an equivalence
\[(X\otimes R)[x^{-1}] \simeq X\otimes R[x^{-1}]\]
for any connected $X$. This generalizes \cite[4.12]{stonek-thhku}, which was only for $X=S^1$.

Thus, from the knowledge of $X\otimes Mf$ we can determine $X\otimes (Mf)[x^{-1}]$ for any $x\in \pi_*(Mf)$. As an important example one can consider the presentation given by Snaith \cite{snaith81} for $KU$, namely $KU\simeq \bS[K(\Z,2)][x^{-1}]$ for a certain $x\in \pi_2 \bS[K(\Z,2)]$. We obtain
\[X\otimes KU \simeq KU \sma \bS[X\odot K(\Z,2)],\]
a result related to \cite{stonek-thhku}, see \cref{ex-s0}. Since the similar equivalence of Snaith $MUP \simeq \bS[BU][x^{-1}]$ for a certain $x\in \pi_2 \bS[BU]$ is not an equivalence of $E_\infty$-ring spectra \cite{hahn-yuan}, we cannot proceed as straightforwardly in this case; however, we can still conclude that $THH(MUP) \simeq MUP \sma SU_+$ as spectra,
see \cref{ex-mup2}. On the other hand, considering $MUP$ as an $E_\infty$-Thom spectrum gives us $THH(MUP)\simeq MUP \sma U_+$. This gives an indirect proof that
\[MUP \sma SU_+ \simeq MUP \sma U_+\]
as spectra, i.e. $U$ and $SU$ have isomorphic unreduced $MUP$-homology.

\subsection{Notations and conventions}
Following \cite{htt}, we call \emph{$\infty$-category} a simplicial set such that every inner horn has a filler. Categories will be considered as $\infty$-categories via the nerve functor. If $\C$ is an $\infty$-category and $c,c'$ are objects in $\C$, then we denote by $\Map_\C(c,c')$ the space of arrows from $c$ to $c'$ in $\C$. An $\infty$-category is \emph{pointed} if it has a zero object. We will call a map that factors through the zero object a \emph{trivial map}. A constant functor with value $c$ will be denoted $\{c\}$.

We denote by $\S$ the $\infty$-category of spaces given by the homotopy coherent nerve of the simplicial category of Kan complexes, and by $\S_*$ its pointed counterpart. Adding a disjoint basepoint gives a left adjoint $(-)_+:\S\to \S_*$ to the forgetful functor. %

The $\infty$-category of spectra will be denoted by $\Sp$. It is a closed symmetric monoidal $\infty$-category \cite[4.8.2.19]{ha}. The internal mapping spectrum will be denoted by $\Sp(A,B)$, the smash product of spectra will be denoted by $A \sma B$ and its monoidal unit, the sphere spectrum, by $\bS$. %

An \emph{$E_\infty$-ring spectrum} $R$ is a commutative algebra object in $\Sp$. The $\infty$-category of left $R$-modules will be denoted by $\Mod_R$; it is a symmetric monoidal $\infty$-category with monoidal product $\sma_R$ and monoidal unit $R$ \cite[4.5.2.1]{ha}. Commutative algebra objects therein are \emph{$E_\infty$-$R$-algebras} and are the objects of the $\infty$-category $\CAlg_R$.

We reserve the notation $\otimes$ for the monoidal product in the $\infty$-category of presentable $\infty$-categories (to be introduced below), and for tensors of spaces with objects of an $\infty$-category. A monoidal product in a general monoidal $\infty$-category will be denoted by $\boxtimes$.

\subsection{Acknowledgments} The authors would like to thank Aras Ergus for pointing out an error in  \cref{section-thomisonots0} in 
a previous version of this paper, Tyler Lawson for allowing us to use an argument of his from MathOverflow in \cref{ex-picmup}, Tobias Barthel and Hongyi Chu for our helpful conversations, and the referee for their careful reading and thoughtful comments, in particular for noticing that there was a superfluous hypothesis in \cref{lemma-basechange} and sharing the more general proof with us. We would also like to thank the Max Planck Institute for Mathematics for its hospitality and financial support.

\section{Presentable \texorpdfstring{$\infty$-categories}{infty-categories} and tensors} \label{section-presentable}

The $\infty$-category $\prl$ of presentable $\infty$-categories \cite[5.5]{htt} is a useful tool when it comes to formulating the idea of an $\infty$-category tensored over a symmetric monoidal $\infty$-category. We review this theory, then we turn to (reflective) localizations of $\infty$-categories. We prove an $\infty$-categorical analogue of Day's reflection theorem, which gives equivalent conditions under which the localization of a closed symmetric monoidal $\infty$-category is closed symmetric monoidal: these conditions are automatically satisfied when the localization is smashing. We then turn our attention to smashing localizations of $\prl$, reviewing the theory of \cite[4.8]{ha} and \cite{ggn}. Some important examples of smashing localizations $L$ of $\prl$ are given by the $\infty$-categories of presentable $\infty$-categories which are pointed, semiadditive, additive or stable: we look at the tensors appearing in these situations. Along the way, we make some contributions to the general theory, like proving that the free $L$-local presentable $\infty$-category on a small simplicial set $K$ is given by the category of $L\S$-valued presheaves on $K$.

\subsection{Generalities} 

Following \cite[4.8]{ha}, \cite{ggn} and \cite[2.2]{abg}, we will work with the closed symmetric monoidal $\infty$-category $\prl$ of presentable $\infty$-categories. The colimit-preserving functors $\C\to \D$ (since the $\infty$-categories are presentable, they coincide with the left adjoint functors) are assembled into a presentable $\infty$-category $\FunL(\C,\D)$. These provide the internal homs to $\prl$. %
The mapping spaces are given by their maximal subspaces. The monoidal product of $\prl$ is denoted $\otimes$, and is characterized by the fact that left adjoint functors out of $\C\otimes \D$ are given by functors out of $\C\times \D$ which preserve colimits separately in each variable. The $\infty$-category $\C\otimes \D$ is also canonically equivalent to $\FunR(\C^\op,\D)$ (where the $R$ denotes that the functors are right adjoints). The monoidal unit of $\prl$ is $\S$, the $\infty$-category of spaces. A commutative algebra in $\prl$ is equivalently a presentable closed symmetric monoidal $\infty$-category $\C$, whose monoidal product we will typically denote by $-\boxtimes -$, its monoidal unit by $\1$, and its internal hom by $\C(-,-)$.

\begin{defn} \label{def-tensored} Let $\R$ be a presentable closed symmetric monoidal $\infty$-category. A presentable $\infty$-category $\C$ is \emph{tensored over $\R$} if it is a module over $\R$ in $\prl$. In particular, there is a functor, the \emph{tensor},
\[- \otimes - : \R\times \C\to \C\]
which preserves colimits separately in each variable. 
\end{defn}

\begin{rmk} \label{rmk-cotensor-enrichment}
 Because we are working in $\prl$, a functor preserves colimits if and only if it is a  left adjoint. This means we also have a {\it cotensor},
 $$ (-)^{(-)}: \R^{\op} \times \C \to \C$$
 and an \emph{enrichment} \cite[7.4.13]{gepner-haugseng},
 $$\R(-,-): \C^{\op} \times \C \to \R .$$
 Thus, an $\infty$-category $\C$ tensored over $\R$ is cotensored and enriched over $\R$. Since we will mostly work with tensors, we have chosen to emphasize them in the previous definition. Note that a morphism of $\R$-modules is a colimit-preserving functor that preserves tensors, cotensors and the enrichment.
\end{rmk}

We will denote the $\infty$-category of presentable $\infty$-categories enriched over $\R$ by $(\prl)^\R$. The previous remark gives us a
fully faithful embedding 
$$\Mod_\R(\prl) \hookrightarrow (\prl)^\R.$$

\begin{rmk} \label{rmk-iterated} Let $\C$ be a presentable $\infty$-category tensored over $\R$. Then for any objects $u,v$ in $\R$ and $c$ in $\C$, by manipulating adjunctions we obtain a natural equivalence $(u\boxtimes v) \otimes c \simeq u \otimes (v\otimes c)$, where $\boxtimes$ is the monoidal product of $\R$. %
\end{rmk}

\begin{ex} \label{ex-tensored-itself} In a symmetric monoidal $\infty$-category, any commutative algebra is canonically a module over itself, with action given by multiplication. In particular, any presentable closed symmetric monoidal $\infty$-category is tensored over itself, with tensor given by the monoidal product and cotensor and enrichment given by the internal hom.
\end{ex}

Since $\S$ is the monoidal unit in $\prl$, it is canonically a commutative algebra object in $\prl$ and every presentable $\infty$-category $\C$ is uniquely a module over it. The action is given by a functor $\S\otimes \C\to \C$ (it is an equivalence) 
whose transpose is a functor $\C\to \FunL(\S,\C)$ which takes the object $c$ to a colimit-preserving functor $F_c:\S\to \C$ such that $F_c(*)=c$. If $X$ is a space, then
\begin{equation}\label{eq-Fc}F_c(X)\simeq F_c(\colim(X \stackrel{\{*\}}{\to} \S)) \simeq \colim (X\stackrel{\{*\}}{\to} \S \stackrel{F_c}{\to} \C) = \colim (X\stackrel{\{c\}}{\to} \C).\end{equation}

\begin{defn}\cite[4.4.4.9]{htt} \label{def-tensorwithspace}Let $\C$ be a cocomplete $\infty$-category and $X$ be a space. We define $X\otimes c$, the \emph{tensor} of $c$ with $X$, as
\begin{equation}\label{def-tensor}
X\otimes c \coloneqq \colim (X\stackrel{\{c\}}{\to} \C).\end{equation}
\end{defn}

In particular, $*\otimes c\simeq c$. In conclusion, the tensor of a presentable $\infty$-category $\C$ over $\S$ is a functor
\begin{equation}\label{eq-tensor}-\otimes -: \S\times \C\to \C\end{equation}
which preserves colimits separately in each variable. %
It satisfies
\begin{equation}\label{adj-tensor}\Map_\C(X\otimes c,d)\simeq \Map_\S(X,\Map_\C(c,d))\end{equation}
for all spaces $X$ and objects $c,d$ in $\C$. Indeed, the right adjoint to $-\otimes c:\S\to \C$ is immediately identified to be $\Map_\C(c,-)$ by considering the adjunction equivalence for $X=*$. Observe that the notation $\otimes$ is being used for two different notions, namely the monoidal product of two presentable $\infty$-categories and the tensor of an object in a presentable $\infty$-category with a space.

\begin{rmk} \label{rmk-kanfibrant}
The colimit formula (\ref{def-tensor}) makes sense for any simplicial set $K$, so one can define $K\otimes c$ for any $c\in \C$. In this case, considering $*\in \S$, we get that $K\otimes *$ is a Kan fibrant replacement of $K$, i.e. a left adjoint of the inclusion of Kan complexes into simplicial sets evaluated at $K$. %
We will denote $K\otimes *$ by $K^\Kan$.
\end{rmk}

We will need the following result on the behavior of tensors in over-categories.

\begin{lemma} \label{lemma-over} Let $\C$ be a presentable $\infty$-category. Let $f:G\to K$ and $g:H\to K$ be morphisms in $\C$ and let $X$ be a space. The equivalence
\[\Map_\C(X\otimes G,H)\simeq \Map_\S(X,\Map_\C(G,H))\]
of spaces restricts to an equivalence
\[\Map_{\C_{/K}}(X\otimes G,H)\simeq \Map_\S(X,\Map_{\C_{/K}}(G,H)) \]
where $X\otimes G$ is considered as an object in $\C_{/K}$ via the morphism $\xymatrix{X\otimes G \ar[r]^-{\ast \otimes \id} &\ast \otimes G \simeq G \ar[r]^-f & K }$.
\begin{proof}
Let $\{ f \}: X \to \C_{/K}$ be the constant diagram mapping $X$ to the object $f$ in $\C_{/K}$. By definition, the colimit of $\{ f\}$ satisfies 
\[\Map_{\C_{/K}}(\colim \{ f \},H)\simeq \Map_\S(X,\Map_{\C_{/K}}(G,H)). \]
By \cite[1.2.13.8]{htt}, colimits in $\C_{/K}$ are created via the standard projection $\pi_K: \C_{/K} \to \C$, which directly implies that 
\[\pi_K(\colim \{ f\}) \simeq \colim \pi_K\{ f \} \simeq \colim \{ G \} = X \otimes G.\]
To finish the proof we need to determine the map $X \otimes G \to K$. 
\par 
Notice it suffices to do it for the case $K = G$. Indeed, we have a projection $\pi_f:(\C_{/K})_{/f} \to \C_{/K}$, which is equivalent to the post-composition map 
$f_!: \C_{/G} \to \C_{/K}$. Moreover, the map $\{f\}: X \to \C_{/K}$ lifts to the map $\{ \id_G\} : X \to \C_{/G}$. 
\begin{center}
 \begin{tikzcd}[row sep=0.5in, column sep=0.5in]
  & \C_{/G} \arrow[r, "\sim"] \arrow[dd, "\pi_G" near start] \arrow[dr, "f_!"] & (\C_{/K})_{/f} \arrow[d, "\pi_f"] \\
  X \arrow[ur, "\{\id_G\}"] \arrow[rr, "\{f\}" near start] \arrow[dr, "\{G\}"']& & \C_{/K} \arrow[dl, "\pi_K"]\\
  & \C & 
 \end{tikzcd}
\end{center}
Here the equivalence $\C_{/G} \xrightarrow{ \ \simeq \ } (\C_{/K})_{/f}$ follows from \cite[4.1.1.7]{htt}.
Thus we want to show that the colimit of 
$\{ \id_G \}$ in $\C_{/G}$ is the map $X \otimes G \to * \otimes G \simeq G$. However, this follows immediately from the fact that for any map $\Delta[n] \to G$ the precomposition map  
$\Delta[n] \otimes G \to X \otimes G \to * \otimes G$ has to be $* \otimes \id$ by definition of being a colimiting cocone.
\end{proof}
\end{lemma}

We now finish this section with an observation about tensors with spaces which have an action of a topological group. If $\C$ is an $\infty$-category and $G$ is a topological group, then one can consider the $\infty$-category of objects of $\C$ with $G$-action, which is $\Fun(BG,\C)$. If $\C$ is a presentable $\infty$-category tensored over a presentable closed symmetric monoidal $\infty$-category $\R$, then by functoriality of the tensor, whenever $X\in \R$ or $c\in \C$ have a $G$-action, then so does $X\otimes c\in \C$, and this is a functorial construction. For example, if $X:BG\to \R$, then the composite $BG \xrightarrow{X} \R \xrightarrow{-\otimes c} \C$ gives $X\otimes c:BG\to \C$. To recover the underlying object of $\C$, precompose the functor with the unique arrow $e: *\to BG$.

Let us take $\R$ to be $\S$. Consider $G$ with its regular $G$-action: we can describe it as %
the left Kan extension of $*\xrightarrow{\{*\}} \S$ along $e:*\to BG$. To avoid confusion, let us denote the resulting functor by $\underline{G}:BG\to \S$. We now claim that, if $c\in \C$, then $\underline{G}\otimes c$ is the free object of $\C$ with $G$-action on $c$. More precisely, we claim:
\begin{prop} \label{prop-bg} For any cocomplete $\infty$-category $\C$ and topological group $G$, there is an adjunction
\begin{equation} \label{eq-bg}\adj{\C\simeq \Fun(*,\C)}{\Fun(BG,\C)}{\underline{G} \otimes -}{e^*}.\end{equation}
\end{prop}
This is a generalization of \cite[IV.2.2]{nikolaus-scholze}, which states (without proof) that, for an $E_\infty$-ring spectrum $A$, the map $A\to S^1\otimes A$ is initial among maps from $A$ to an $E_\infty$-ring spectrum with an $S^1$-action.

To prove the proposition, it suffices to identify $\underline{G}\otimes -$ as the left Kan extension functor along $e$ \cite[4.3.3.7]{htt}. We first prove a general lemma:

\begin{lemma} \label{lemma-kan}
Consider the following diagram of left Kan extensions in cocomplete $\infty$-categories:
 \[
 \begin{tikzcd}[row sep=0.2in, column sep=0.3in]
 \C \arrow[rr, "f", ""{name=V, below}] \arrow[dr, "g", pos=0.6,  ""{name=U, below}] \arrow[dd, "h"'] & & \D \\
 & \E \arrow[ur, "\Lan_gf"'] & \\
 \mathscr{F} \arrow[ur, "\Lan_hg"']& & 
 \arrow[Rightarrow, shorten <= 0.35cm, shorten >= .4cm, "\tau", swap, from=U, to=3-1]
 \arrow[Rightarrow, shorten <= 0.1cm, shorten >= .05cm, from=V, to=2-2]
 \end{tikzcd}
 \]
 Suppose that for every $c\in\C$ and $x\in\mathscr{F}$ the following composition is an equivalence of spaces:
 \begin{equation}\label{eq-kan}\xymatrix{\Map_{\mathscr{F}}(h(c),x) \ar[r]^-{\Lan_hg} & \Map_\E(\Lan_hg(h(c)),\Lan_hg(x)) \ar[r]^-{\tau_c^*} & \Map_\E(g(c),\Lan_hg(x)). }\end{equation}
 Then the universal natural transformation
 \[\Lan_hf \to \Lan_gf \circ \Lan_hg\]
 is an equivalence. 
\end{lemma}

\begin{proof}
 Using the colimit formula for left Kan extensions %
 we get, for $x\in \mathscr{F}$: 
 \[\Lan_hf(x) \simeq \colim(h \downarrow x \to \C \xrightarrow{ \ f \ } \D), \text{ and}\]
 \[(\Lan_gf \circ \Lan_hg)(x) \simeq \colim(g \downarrow \Lan_hg(x) \to \C \xrightarrow{ \ f \ } \D).\]
The natural map $\Lan_hf(x)\to (\Lan_gf\circ \Lan_hg)(x)$ is induced by the natural map between the comma $\infty$-categories $h \downarrow x \to g \downarrow \Lan_hg(x)$ described in (\ref{eq-kan}). To see that the former is an equivalence, it suffices to see that the latter is cofinal. %
 However, both comma $\infty$-categories are the domain of a corresponding right fibration over $\C$. In this case, being cofinal is equivalent to being a fiber-wise equivalence of spaces  \cite[2.2.3.13, 4.1.2.5]{htt}. This is precisely the condition that (\ref{eq-kan}) be an equivalence.
\end{proof}

Before proceeding to the proof of \cref{prop-bg}, let us observe that the map $\Lan_hf\to \Lan_gf \circ \Lan_hg$ is not necessarily an equivalence without extra conditions like (\ref{eq-kan}). First, note that for two spaces $X$ and $Y$, the left Kan extension 
\begin{center}
	\begin{tikzcd}
		* \arrow[r, "\{Y\}"] \arrow[d, "\{X\}"'] & \S \\
		\S \arrow[ur, "\Lan_{\{X\}}\{Y\}"'] 
	\end{tikzcd}
\end{center}
is given by $\Lan_{\{X\}}\{Y\}(Z) \simeq \Map(X,Z) \times Y$.
Now take 
$\C=*$, $\D=\E=\mathscr{F}=\S$, $f=\{*\}$, $g=\{S^0\}$, and $h=\{\emptyset\}$. In this case, $\Lan_hf\simeq \{*\}$, and $\Lan_gf \circ \Lan_hg \simeq \{S^0\times S^0\}$.

\begin{proof}[Proof of \cref{prop-bg}]
It now suffices to observe that we have a diagram of left Kan extensions as follows,
\[
\begin{tikzcd}[row sep=0.2in, column sep=0.4in]
\ast \arrow[rr, "\{c\}", ""{name=V, below}] \arrow[dr, "\{*\}", pos=0.55,  ""{name=U, below}] \arrow[dd, "e"'] & & \C \\
& \S \arrow[ur, "-\otimes c"'] & \\
BG \arrow[ur, "\underline{G}"']& & 
\arrow[Rightarrow, shorten <= 0.35cm, shorten >= .5cm, from=U, to=3-1]
\arrow[Rightarrow, shorten <= 0.1cm, shorten >= .05cm, from=V, to=2-2]
\end{tikzcd}
\]
and that (\ref{eq-kan}) amounts to $G\xrightarrow{\id} G$ in this case.
\end{proof}

\subsection{Localizations}

We quickly review some definitions and results about (reflective) localizations, mostly from \cite[5.2.7]{htt}, \cite[4.8.2]{ha} and \cite{ggn}. Then we prove a version of Day's reflection theorem.

\begin{defn} Let $\C$ be an $\infty$-category. A functor $L:\C\to \D$ is a \emph{localization} if it admits a fully faithful right adjoint $i$. The $\infty$-category $\D$ is equivalent via $i$ to a full subcategory of $\C$ denoted $L\C$, so we often write $L:\C\to L\C$ (or even $L:\C\to \C$) and neglect to mention $i$. The objects of $L\C$ are \emph{$L$-local}. For any $c\in \C$, there is a \emph{localization map} $c\to Lc$ given by the unit of the adjunction.\end{defn}

For any localization $L$, there are natural equivalences $Lc\to LLc$ for $c\in \C$. An object $c\in \C$ is in $L\C$ if and only if the localization map $c\to Lc$ is an equivalence, if and only if for every $c'\in \C$ the localization map $c'\to Lc'$ induces an equivalence of spaces
\[\Map_\C(Lc',c)\simeq \Map_\C(c',c).\]

The following theorem %
is an $\infty$-categorical analogue of Day's reflection theorem \cite{day-reflection}. We remind the reader that $\C(c,d)\in \C$ denotes an internal hom.

\begin{thm} \label{prop-dayreflection}Let $\C$ be a closed symmetric monoidal $\infty$-category. Let $L:\C\to L\C$ be a localization functor. The following are equivalent:
\begin{enumerate}
\item For all $c\in \C$, $d\in L\C$, the localization map $\C(c,d)\to L\C(c,d)$ is an equivalence,
\item For all $c\in \C$, $d\in L\C$, the localization map $c\to Lc$ induces an equivalence $\C(Lc,d)\to \C(c,d)$,
\item For all $c,c'\in \C$, the localization map $c\to Lc$ induces an equivalence $L(c\boxtimes c')\to L(Lc\boxtimes c')$,
\item For all $c,c'\in \C$, the localization maps of $c$ and $c'$ induce an equivalence $L(c\boxtimes c')\to L(Lc\boxtimes Lc')$.
\end{enumerate}
When these equivalent conditions are satisfied, $L\C$ admits a closed symmetric monoidal structure such that $L$ is symmetric monoidal and its right adjoint $i$ is lax symmetric monoidal and closed, i.e. the internal hom in $L\C$ is given by $\C(d,d')$ for all $d,d'\in L\C$. %
\begin{proof}
($2 \Leftrightarrow 3$) Let $c,c'\in \C$, $d\in L\C$. The following diagram commutes,
\[\xymatrix@C-.3cm@R-.3cm{
\Map_{L\C}(L(Lc\boxtimes c'),d) \ar[r] \ar[d]_-\simeq &
\Map_{L\C}(L(c\boxtimes c'),d) \ar[d]^-\simeq \\
\Map_\C(Lc\boxtimes c',d) \ar[d]_-\simeq &
\Map_\C(c\boxtimes c',d) \ar[d]^-\simeq\\
\Map_\C(c',\C(Lc,d)) \ar[r] & \Map_\C(c',\C(c,d))
}\]
so the top arrow is an equivalence if and only if the bottom one is an equivalence, and the Yoneda lemma finishes the proof.

($1\Rightarrow 3$) Let $c,c'\in \C$ and $d\in L\C$. The following diagram commutes,
\[\xymatrix@C-.2cm@R-.3cm{
\Map_{L\C}(L(Lc\boxtimes c'),d) \ar[r] \ar[d]_-\simeq &
\Map_{L\C}(L(c\boxtimes c'),d) \ar[d]_-\simeq \\
\Map_\C(Lc\boxtimes c',d) \ar[d]_-\simeq &
\Map_\C(c\boxtimes c',d) \ar[d]_-\simeq \\
\Map_\C(Lc, \C(c',d)) \ar[d]_-\simeq &
\Map_\C(c,\C(c',d)) \ar[d]_-\simeq \\
\Map_\C(Lc,L\C(c',d)) \ar[r]^-\sim &
\Map_\C(c,L\C(c',d))
}\]
so the top horizontal map is an equivalence, and the Yoneda lemma finishes the proof.

($3\Rightarrow 4$)  Let $c,c'\in \C$. The following diagram commutes,
\[\xymatrix@C-.6cm@R-.3cm{
L(c\boxtimes c') \ar[rr] \ar[rd]_-\simeq && L(Lc\boxtimes Lc') \\ & L(Lc\boxtimes c') \ar[ru]_-\simeq
}\]
so the horizontal map is an equivalence.

($4 \Rightarrow 1$) Let us denote by $\eta$ all localization maps. Let $c\in \C$, $d\in L\C$. We will construct an inverse to $\eta:\C(c,d)\to L\C(c,d)$. First, note that if $\nu:L\C(c,d)\to \C(c,d)$ is a left inverse to $\eta$, then it is also a right inverse to it. Indeed, in this case, both $\id$ and $\eta\circ \nu$ can play the role of the dotted arrow in the following diagram, making it commute:
\[\xymatrix{
\C(c,d) \ar[r]^-\eta \ar[d]_-\eta & L\C(c,d) \\ L\C(c,d) \ar@{.>}[ru]
}\]
whence by the universal property of $\eta$ we deduce that $\eta \circ \nu$ is equivalent to $\id$.

By adjunction, constructing a left inverse $\nu$ to $\eta$ is equivalent to constructing an arrow $\overline\nu:L\C(c,d) \boxtimes c\to d$ making the following diagram commute:
\[\xymatrix{
\C(c,d)\boxtimes c \ar[r]^-e \ar[d]_-{\eta \boxtimes \id} & d \\ L\C(c,d)\boxtimes c. \ar[ru]_-{\overline \nu}
}\]
Here $e:\C(c,d)\boxtimes c \to d$, the evaluation map, is transpose to $\id:\C(c,d)\to \C(c,d)$. Consider the following commutative diagram:
\[\xymatrix{
\C(c,d)\boxtimes c \ar[dd]_-{\eta \boxtimes \id} \ar[rr]^-e \ar[rrd]^-\eta \ar[rdd]^-{\eta \boxtimes \eta} &&
d \\
&& L(\C(c,d)\boxtimes c) \ar[d]^-{L(\eta \boxtimes \eta)} \ar[u]_-u \\
L\C(c,d)\boxtimes c \ar[r]_-{\id\boxtimes \eta} &
L\C(c,d)\boxtimes Lc \ar[r]_-\eta &
L(L\C(c,d) \boxtimes Lc)
}\]
where $u$ exists by the universal property of $\eta$. By hypothesis, the vertical map $L(\eta \boxtimes \eta)$ admits an inverse $f$. Define $\overline \nu$ to be $u\circ f \circ \eta \circ (\id \boxtimes \eta)$: it is the $\overline{\nu}$ we were looking for.\\

We now prove the last assertion. Call an arrow $f:c\to c'$ in $\C$ a \emph{local equivalence} if $Lf$ is an equivalence. Note that if $f$ is a local equivalence, then $f\boxtimes \id:c\boxtimes c''\to c'\boxtimes c''$ is a local equivalence for any $c''\in \C$. Indeed, in the following commutative diagram, the vertical arrows are equivalences by (3):
\[\xymatrix@C+2pc{
L(c\boxtimes c'') \ar[r]^-{L(f\boxtimes \id)} \ar[d]_-{L(\eta \boxtimes \id)}^-\simeq &
L(c'\boxtimes c'') \ar[d]^-{L(\eta \boxtimes \id)}_-\simeq \\ L(Lc\boxtimes c'') \ar[r]_-{L(Lf\boxtimes \id)}^-\sim & L(Lc' \boxtimes c'')
}\]
Now \cite[3.4]{ggn} applies to prove that $L\C$ has the desired symmetric monoidal structure, $L$ is symmetric monoidal and $i$ is lax symmetric monoidal. 

We now prove that the symmetric monoidal structure on $L\C$ is closed with internal hom given by the internal hom of $\C$. Let $d\in L\C$. We have the following adjunctions 

\[\begin{tikzcd}[row sep=0.5in, column sep=0.5in]
  \C \arrow[r, "- \boxtimes d", shift left = 0.05in] & \C \arrow[l, "\C(d \comma -)", shift left=0.05in] \arrow[r, "L", shift left = 0.05in] & L\C.  \arrow[l, "i", shift left=0.05in]
 \end{tikzcd}\]
We now compose the adjunctions and notice that the right adjoint $\C(d,-) \circ i \simeq \C(d,-)$ takes values in the subcategory $L\C$ by (1), which gives us an adjunction 
\[
 \begin{tikzcd}[row sep=0.5in, column sep=0.5in]
  L\C \arrow[r, "L(- \boxtimes d)", shift left = 0.05in] & L\C. \arrow[l, "\C(d \comma -)", shift left=0.05in]
 \end{tikzcd}\]
Since $d$ is local and $L$ is symmetric monoidal, we get $L(- \boxtimes d) %
\simeq - \boxtimes_{L\C} d$, giving us the desired adjunction.
\end{proof}
\end{thm}

We now consider a strong condition one can impose on localizations:

\begin{defn}
A localization $L:\C\to L\C$ is \emph{smashing} if $\C$ is a symmetric monoidal $\infty$-category and $L$ is of the form $ - \boxtimes I$ for some \emph{smashing object} $I$ of $\C$.
\end{defn}

Note that $I$ is equivalent to the localization of the monoidal unit of $\C$, and that for any $c,c'\in \C$, we have
\begin{equation}\label{eq-boxy}
c\boxtimes Lc'\simeq L(c\boxtimes c') \simeq Lc\boxtimes c'.
\end{equation}

\begin{prop} \label{prop-smashing-day} If $\C\to L\C$ is a smashing localization functor of a closed symmetric monoidal $\infty$-category, then the equivalent conditions of \cref{prop-dayreflection} are satisfied. Moreover, the monoidal product $\boxtimes_{L\C}$ of $L\C$ can be computed in $\C$, i.e. if $i$ denotes the right adjoint to $L$, then
\[i(d\boxtimes_{L\C} d')\simeq id\boxtimes id'.\]
\end{prop}
Note that we are not saying that $i$ is symmetric monoidal: indeed, $i$ typically does not preserve the monoidal unit.
\begin{proof}
In this situation, condition (3) of \cref{prop-dayreflection} is immediate, and the statement about $\boxtimes_{L\C}$ readily follows from (\ref{eq-boxy}).
\end{proof}

\subsection{Smashing localizations of $\prl$}

We now look at smashing localizations $L$ of $\prl$. From \cref{prop-smashing-day} , we get that if $A$ is a commutative algebra in $\prl$, then we have a commutative algebra $LA$ in $\prl$. Recall that $\infty$-categories of modules over a commutative algebra can be endowed with a closed symmetric monoidal structure \cite[4.5.2]{ha}.

\begin{thm}\cite[4.8.2.10]{ha}, \cite[3.8, 3.9]{ggn} \label{thm-ggn-main} Let $L:\prl\to\prl$ be a smashing localization.
\begin{enumerate}[(i)]
\item The underlying presentable $\infty$-category of an object in $\Mod_{L\S}(\prl)$ is $L$-local, and the forgetful functor $\Mod_{L\S}(\prl)\to L\prl$ is an equivalence of symmetric monoidal $\infty$-categories, where $L\prl$ is a symmetric monoidal $\infty$-category as in \cref{prop-dayreflection}. %
\item For any presentable closed symmetric monoidal $\infty$-category $\C$, the $\infty$-category $L\C$ admits a unique closed symmetric monoidal structure such that the localization map $\C\to L\C$ gets a symmetric monoidal structure.
\item Given a second smashing localization $L':\prl\to\prl$ such that $L'\prl\subset L\prl$, the induced morphism $\eta_\C:L\C\to L'\C$ admits a unique symmetric monoidal structure.
\end{enumerate}
\end{thm}
In (iii), the morphism $L\C\to L'\C$ is obtained as follows. Let $j:L'\prl\to L\prl$ denote the inclusion. The unit of the $(L',i')$ adjunction gives a map $\C\to i'L'\C=ijL'\C$ whose transpose under the $(L,i)$ adjunction is the desired map \begin{equation}\label
{eq-eta}\eta_\C:L\C\to jL'\C=L'\C.\end{equation}

\begin{rmk} \label{rmk-functor-tensor}
 Let $\C$ and $\D$ be closed symmetric monoidal $\infty$-categories. Note that if $F:\C\to \D$ is a symmetric monoidal functor which is an equivalence of $\infty$-categories, then it preserves the internal hom, i.e. if $c,c'\in \C$ then $F$ gives an equivalence $F\C(c,c') \stackrel{\sim}{\to} \D(Fc,Fc')$. Indeed, let $c,c'\in\C$ and $d\in\D$. Then $d\simeq Fz$ for some $z\in\C$, and
 \begin{align*}
 	\Map_{\D}(d,&\D(Fc,Fc')) 
 						   \simeq \Map_{\D}(Fz\boxtimes_{\D}Fc,Fc') \simeq
 						    \Map_{\D}(F(z\boxtimes_{\C}c),Fc') \simeq \\
 						   &\simeq \Map_{\C}(z\boxtimes_{\C}c,c')  
 						   \simeq \Map_{\C}(z,\C(c,c'))
 						   \simeq \Map_{\D}(d,F\C(c,c'))
 \end{align*}
 as desired.
Since the internal hom in $L\prl$ is given by $\FunL$ (\cref{prop-dayreflection}) %
, from \cref{thm-ggn-main}(i) we deduce that for $\C,\D\in \Mod_{L\S}(\prl)$, there is an equivalence of $\infty$-categories
\[\Mod_{L\S}(\prl)(\C,\D)\stackrel{\sim}{\to} \FunL(\C,\D).\]
Thus, any colimit-preserving functor $\C\to \D$ can naturally be given the structure of an $L\S$-module map.
\end{rmk}

If $f:A\to B$ is a morphism of commutative algebras in a presentable closed symmetric monoidal $\infty$-category $\C$, then there is a restriction of scalars functor $\res_f:\Mod_B(\C)\to \Mod_A(\C)$ with left adjoint given by the extension of scalars functor $B\boxtimes_A-:\Mod_A(\C)\to \Mod_B(\C)$ \cite[4.5.3]{ha}.

\begin{prop} \label{prop-res} Let $L,L':\prl\to \prl$ be two smashing localizations such that $L'\prl\subset L\prl$. The following diagrams commute:
\[\xymatrix{\Mod_{L'\S}(\prl) \ar[r]^-\sim & L'\prl \\ \Mod_{L\S}(\prl) \ar[r]_-\sim \ar[u]^-{L'\S\otimes-} & L\prl \ar[u]_-{L'} } \hspace{1cm}
\xymatrix{\Mod_{L'\S}(\prl) \ar[r]^-\sim \ar[d]_-{\res_f} & L'\prl \ar@{^(->}[d] \\ \Mod_{L\S}(\prl) \ar[r]_-\sim & L\prl. }
\]
In particular, if $\C\in L'\prl$, then $\C$ is tensored over $L\S$: the tensor of $c\in \C$ with $A\in L\S$ is given by \begin{equation}\label{eq-tensorextend} \eta_\S(A)\otimes_{\C/L'\S}c,\end{equation} where $\eta_\S:L\S\to L'\S$ was described in (\ref{eq-eta}) and $\otimes_{\C/L'\S}$ denotes the tensor of $\C$ over $L'\S$.
\begin{proof} 
 The vertical maps in the two squares are adjoints (where the ones on the left square are the left adjoints) and thus it suffices to prove the square on the left commutes.
 This follows immediately from the fact that $L'$ is smashing and thus $L' \simeq L'\S \otimes -$.
\end{proof}
\end{prop}

\begin{rmk} \label{rmk-dependence} From (\ref{eq-tensorextend}) we deduce that the tensor with $A\in L\S$ depends only on $\eta_\S(A)$, in the sense that if $A, A'\in L\S$ are such that $\eta_\S(A)\simeq \eta_\S(A')$, then the tensors with $A$ or with $A'$ are equivalent.
\end{rmk}

There is an analogous result for enrichments under the same hypotheses as the above proposition. First note that since the localization map $\eta:L\S\to L'\S$ is a symmetric monoidal functor, then its right adjoint $\bar{\eta}: L'\S \to L\S$ is lax monoidal \cite[A.5.11]{gepner-haugseng}, so it gives us a functor
 $(\prl)^{\bar{\eta}}:  (\prl)^{L'\S} \to (\prl)^{L\S}$ between $\infty$-categories of presentable enriched $\infty$-categories \cite[5.7.8]{gepner-haugseng}.

\begin{prop} \label{prop-enrichments}
In the situation of the previous proposition, the following diagram commutes:
 \begin{center}
  \begin{tikzcd}[row sep=0.5in, column sep=0.5in]
   \Mod_{L'\S}(\prl) \arrow[r, hookrightarrow] \arrow[d, "\res"] & (\prl)^{L'\S} \arrow[d, "(\prl)^{\bar{\eta}}"] \\ 
   \Mod_{L\S}(\prl) \arrow[r, hookrightarrow] & (\prl)^{L\S}.
  \end{tikzcd}
 \end{center}
In particular, if $\C\in L'\prl$, then $\C$ is enriched over $L\S$: the enriched hom $(L\S)(c,c')$ where $c,c'\in \C$ is given by $\bar{\eta}((L'\S)(c,c'))$.
\end{prop}

\begin{proof}
 Let $\C$ be an $L'\S$-module. Then the desired result follows from the fact that the following diagram commutes:
 \begin{center}
  \begin{tikzcd}[row sep=0.5in, column sep=0.5in]
   \FunL(\C \times L'\S,\C) \arrow[r, "\sim"] \arrow[d, "(\id \times \eta)^*"] & \FunR(\C^\op \times \C, L'\S) \arrow[d, "\bar{\eta}_*"] \\
   \FunL(\C \times L\S,\C) \arrow[r, "\sim"] & \FunR(\C^\op \times \C,L\S).
  \end{tikzcd}
 \end{center}
  Indeed, the enrichment map $\C^\op \times \C \to L\S$ is equivalent to the composition \[\xymatrix{\C^\op \times \C \ar[rr]^-{L'\S(-,-)} && L'\S \ar[r]^-{\bar\eta} & L\S.}\qedhere\]
\end{proof}

\begin{defn} \label{def-freegen}
 Let $\P$ be a full subcategory of $\prl$, $\C$ be a presentable $\infty$-category in $\P$ and $c$ be an object of $\C$. Then we say $\C$ is {\it freely generated in $\P$ by $c$} 
 if for any presentable $\infty$-category $\D$ in $\P$ the following functor induced by $\{c\}:*\to \C$ is an equivalence:
 \[\xymatrix{\ev_c: \FunL(\C,\D) \ar[r]^-{\{c\}^*} & \Fun(*,\D) \simeq \D.}\]
\end{defn}

The quintessential example is given by $\S$, which is freely generated in $\prl$ by any contractible space.

\begin{prop}\label{prop-freely-generated} Let $L$ be a smashing localization of $\prl$. Then $L\S$ is freely generated in $L\prl$ by the monoidal unit of $L\S$. %
\begin{proof} Let $\1$ denote the unit of $L\S$ and let $\D\in L\prl$. We can give two proofs. The first one follows from %
\cref{prop-dayreflection}(2), using the localization map $\S\to L\S$: \[\ev_\1:\FunL(L\S,\D)\stackrel{\sim}{\to} \FunL(\S,\D)\stackrel{\sim}{\to} \Fun(*,\D) \simeq \D.\]
For the second proof, first note that by \cref{rmk-functor-tensor}, $\FunL(L\S,\D)\simeq\Mod_{L\S}(\prl)(L\S,\D)$. In this way, $\ev_\1$ is equivalently given by a functor $\Mod_{L\S}(\prl)(L\S,\D) \to \D$, which is the canonical equivalence obtained from the fact that $L\S$ is the monoidal unit of $\Mod_{L\S}(\prl)$.
\end{proof}
\end{prop}

The previous proposition might suggest that $L\S$ behaves similarly to $\S$ itself in certain ways. We will now observe that from two elementary points of view this is not the case,  even in a simple example. We will use that there is a smashing localization $L$ of $\prl$ such that $L\prl$ is given by pointed presentable $\infty$-categories, and $L\S$ is given by $\S_*$. We will explain this in more detail in \cref{subsection-fourloc}.

\begin{rmk} \label{ex-eternal} Left Kan extension along the Yoneda embedding $j:S\to \P(S)$ for a simplicial set $S$ induces an adjunction $\adj{\Fun(S,\D)}{\FunL(\P(S),\D)}{\Lan_j}{j^*}$ \cite[5.1.5]{htt}; this is true in particular for $S=*$.  A general left Kan extension along an arbitrary functor, however, may not preserve colimits, so the inverse to $\ev_c$ in \cref{def-freegen} need not be equivalent to $\Lan_c$.

To see this, consider the example outlined above where $L\S$ is $\S_*$. If the inverse to $\{S^0\}^*$ were $\Lan_{\{S^0\}}$, we would have that for any $X\in \S_*$,  $\Lan_{\{S^0\}}\{X\}(S^0)\simeq X$. However, the colimit formula for $\Lan$ gives us that $\Lan_{\{S^0\}}\{X\}(S^0)\simeq X\vee X$.
\end{rmk}

\begin{rmk}
Note that %
\cref{prop-freely-generated} is not equivalent to saying that every object in $L\S$ is a 
colimit of a constant diagram with value the monoidal unit. 
One might have suspected this because of what happens in the case of $\S$: there, every object $X$ is a colimit of a constant diagram with value the one-point space. Indeed, 
 \[X \simeq X \otimes * = \colim (X \xrightarrow{ \ \{*\} \ } \S).\]
We will now observe that already in the simple case outlined above where $L\S$ is $\S_*$, the analogous statement fails, i.e. we will show that not every object in $\S_*$ is the colimit of a constant diagram with value $S^0$. 

Let $K$ be a simplicial set. Then we have 
 $$\colim( K \xrightarrow{ \ \{S^0\} \ } \S_*) \simeq \colim( K \to * \xrightarrow{ \ \{*\} \ } \S  \xrightarrow{ \ (-)_+ \ } \S_*) \simeq \colim( K \to * \xrightarrow{ \ \{*\} \ } \S  )_+ \simeq (K^\Kan)_+,$$
 where $K^\Kan$ is the Kan fibrant replacement of $K$ of \cref{rmk-kanfibrant}. Notice that we used that the functor $(-)_+$ is a left adjoint and thus commutes with colimits. 
 Therefore, the colimit of any constant diagram with value $S^0$ necessarily has a disjoint base point, which clearly does not hold for every pointed space.

Note that, at a first glance, one could have thought that for a given pointed space $(Y,y)$ one could use the cofiber sequence 
 $$\{*\}_+ \xrightarrow{ \ \{y\}_+ \ } Y_+ \to Y$$
 to construct $Y$ as the colimit of a constant diagram with value $S^0$, given that $Y_+$ and $\{*\}_+ \simeq S^0$ can both be constructed as such colimits. However, that cofiber sequence is a pushout along the map $S^0 \to *$, so in order to use this argument, 
 we must first be able to construct this map as the colimit of a map of diagrams. More precisely, if $\colim(I \xrightarrow{\{S^0\}} \S_*) \simeq S^0$ and $\colim (J \xrightarrow{\{S^0\}} \S_*) \simeq *$, then we would need to express the map $S^0\to *$ as the colimit of a map of diagrams $I\to J$. %
 However, $ \colim (J \xrightarrow{\{S^0\}} \S_*) \simeq *$ if and only if $J = \emptyset$ and so there cannot be any functor $I \to J$.
\end{rmk}

\cref{prop-freely-generated} can be generalized. %
If $K$ is a small simplicial set, then the $\infty$-category of space-valued presheaves $\P(K)\coloneqq \Fun(K^\op,\S)$ is the free cocomplete $\infty$-category on $K$, in the sense that composition with the Yoneda embedding $j:K\to \P(K)$ induces an equivalence of $\infty$-categories
\begin{equation}\label{eq-freecoco}
\FunL(\P(K),\C) \simeq \Fun(K,\C)
\end{equation}
for any cocomplete $\infty$-category $\C$ \cite[5.1.5.6]{htt}. Since $\P(K)$ is presentable \cite[5.5.3.6]{htt}, then $\P(K)$ can also be regarded as the free presentable $\infty$-category on $K$.

Let $L$ be a smashing localization of $\prl$. Define
\[\P_{L\S}(K)\coloneqq \Fun(K^\op,L\S).\]
We will now prove that $\P_{L\S}(K)$ is the free object in $L\prl$ (equivalently, the free $L\S$-module) on $K$. First, we prove that $L\P(K)\simeq \P_{L\S}(K)$.

\begin{lemma} \label{lemma-lpk} Let $K$ be a small simplicial set. Let $L$ be a smashing localization of $\prl$. There is an equivalence of $\infty$-categories
\begin{equation}\label{eq-lpk} L\Fun(K,\S) \simeq \Fun(K,L\S).\end{equation}
Moreover, this equivalence makes the following triangle commute:
\[\xymatrix{\Fun(K,\S) \ar[r]^-u \ar[rd]_-{v_*} & L\Fun(K,\S) \ar@{-}[d]^-\sim \\ & \Fun(K,L\S)}\]
where $u:\Fun(K,\S)\to L\Fun(K,\S)$ and $v:\S\to L\S$ are the localization maps.
\begin{proof} Recall that if $\C,\D\in \prl$, then $\C\otimes \D \simeq \FunR(\C^\op,\D)$. Using the equivalence of $\infty$-categories $\FunL(\C,\D)^\op\simeq \FunR(\D,\C)$ \cite[5.2.6.2]{htt} and (\ref{eq-freecoco}), we deduce:
\begin{align}
\label{eq-lpkeq}\Fun(K,L\S) &\simeq \Fun(K^\op,(L\S)^\op)^\op \simeq \FunL(\Fun(K,\S),(L\S)^\op)^\op\simeq  \\
& \simeq \FunR((L\S)^\op,\Fun(K,\S)) \simeq L\S \otimes \Fun(K,\S) \simeq L\Fun(K,\S). \nonumber
\end{align}
We will now prove that the triangle commutes.

To see this, first note that the equivalences in (\ref{eq-lpkeq}) are natural in the localization $L$. To illustrate, for a natural transformation $w:L'\Rightarrow L$ of localizations, 
the naturality of the last equivalence takes the form
\[\xymatrix{
L'\Fun(K,\S) \ar[r]^-{w_{\Fun(K,\S)}} \ar[d]_-\sim & L\Fun(K,\S) \ar[d]^-\sim \\ L'\S \otimes \Fun(K,\S) \ar[r]_-{w_\S \otimes \id} & L\S \otimes \Fun(K,\S).
}\]
Taking $L'$ to be the identity and considering the localization natural transformation $\id\Rightarrow L$, the naturality of (\ref{eq-lpkeq}) takes the form of the following commutative square:
\begin{equation} \label{eq-squarefun}\xymatrix{
\Fun(K,\S) \ar[r]^-u \ar[d]_-\sim& L\Fun(K,\S) \ar[d]^-\sim \\ \Fun(K,\S) \ar[r]_{v_*} & \Fun(K,L\S).
}\end{equation}
Therefore, to prove that the triangle commutes it suffices to see that the equivalence on the left of (\ref{eq-squarefun}) is the identity.

Let $F: K \to \S$. The first equivalence of (\ref{eq-lpkeq}) is actually an isomorphism of $\infty$-categories that gives us $F^\op: K^\op \to \S^\op$. The next equivalence
extends $F^\op$ to a functor $\widehat{F^\op}: \Fun(K,\S) \to \S^\op$. In the next step, we associate to $\widehat{F^\op}$ a right adjoint functor $G: \S^\op \to \Fun(K,\S)$. In the last step we simply evaluate to get 
$G(*): K \to \S$. We need to prove that $G(*) \simeq F: K \to \S$. Let $k\in K$, and let $j$ be the Yoneda embedding of $K^\op$. The chain of equivalences
\begin{align*}
G(*)(k) &\simeq \Map_{\Fun(K,\S)}(j(k), G(*)) \simeq \Map_{\S^\op}(\widehat{F^\op}(j(k)), *)  \simeq  \\
&\simeq \Map_{\S^\op}(F^\op(k), *) \simeq \Map_{\S}(*,F^\op(k)) \simeq F^\op(k) \simeq F(k)
\end{align*}
gives us the desired result.
\end{proof}
\end{lemma}

\begin{thm} \label{prop-freelsmod} Let $K$ be a small simplicial set. Let $L$ be a smashing localization of $\prl$. Then $\P_{L\S}(K)$ is freely generated in $L\prl$ by $K$. More precisely, composition with $v_*\circ j:K\to \P_{L\S}(K)$ induces an equivalence of $\infty$-categories
\[\xymatrix@C+1pc{\FunL(\P_{L\S}(K),\C) \ar[r]^-\sim_-{(v_*\circ j)^*} & \Fun(K,\C) } \]
for any $\C\in L\prl$, where $v:\S\to L\S$ is the localization map and $j:K\to \P(K)$ is the Yoneda embedding.
\begin{proof}
Let $u:\P(K)\to L\P(K)$ be the localization map. We have equivalences of $\infty$-categories
\[\xymatrix{\FunL(L\P(K),\C) \ar[r]_-{u^*}^\sim & \FunL(\P(K),\C) \ar[r]^-\sim_-{j^*} & \Fun(K,\C)}\]
and the result now follows from \cref{lemma-lpk}.
\end{proof}
\end{thm}

\subsection{Four localizations of $\prl$} \label{subsection-fourloc}

Let us start by fixing some notation.  If $\C$ is a presentable $\infty$-category, %
let $\C_*$ denote the category of pointed objects of $\C$ (the undercategory $\C_{*/}$ where $*$ is a final object of $\C$),  $\Mon(\C)$ the $\infty$-category of $E_\infty$-monoids in $\C$ (special $\Gamma$-objects), $\Grp(\C)$ the $\infty$-category of $E_\infty$-groups in $\C$ (grouplike $E_\infty$-monoids, or very special $\Gamma$-objects), and $\Sp(\C)$ the stabilization of $\C$. See \cite{ggn} or \cite[2.4.2, 1.4]{ha} for more details. Note that $\Mon(\C)$ is equivalent to $\CAlg(\C)$, where $\C$ is given the cartesian monoidal structure.

In \cite{ggn}, the authors consider the following smashing localizations of $\prl$: tensoring with $\S_*$, $\Mon(\S)$, $\Grp(\S)$ and $\Sp$, and determine descriptions for the corresponding local objects. These are pointed, semiadditive, additive and stable $\infty$-categories, respectively. We display this in the following table, where $\C$ is any presentable $\infty$-category.

 \begin{center}
  \begin{tabular}{|c|c|c|c|c|}
  \hline 
   Categorical property & Full subcat. of $\prl$ & Smashing object $L\S$ & $L\S\otimes \C$ & Unit of $L\S$\\ \hline
   No condition & $\prl$ & $\S$ & $\C$ & $*$ \\ \hline 
   Pointed & $\prlpt$ & $\S_*$ & $\C_*$ & $S^0$ \\ \hline 
   Semiadditive & $\prlsem$ & $\Mon(\S)$ & $\Mon(\C)$ & $\bigsqcup_{n\geq 0} B\Sigma_n$ \\ \hline %
   Additive & $\prladd$ & $\Grp(\S)$ & $\Grp(\C)$ & $\bS$ \\ \hline 
   Stable & $\prlst$ & $\Sp$ & $\Sp(\C)$ & $\bS$ \\ \hline 
  \end{tabular}
 \end{center}
The categorical properties above are listed in increasing order of restrictiveness. We have adopted the terminology of \cite{ha}: in \cite{ggn}, the authors use the adjective ``preadditive'' instead of ``semiadditive''.\\

In \cite[4.10, 5.1]{ggn} the authors obtain the following as a corollary of the result we summarized in \cref{thm-ggn-main}. %
There is a chain of left adjoint functors 
\begin{equation} \label{five}
 \begin{tikzcd}[row sep=0.5in, column sep=0.5in]
  \S \arrow[r, "(-)_+"] & \S_* \arrow[r] \arrow[rrr, bend right = 20, "\Sigma^{\infty}"'] & \Mon(\S) \arrow[r] & \Grp(\S) \arrow[r] & \Sp
 \end{tikzcd}
\end{equation}
given by tensoring with the respective smashing object. The resulting functor $\S_*\to \Sp$ is $\Sigma^\infty$. Each of these $\infty$-categories admits a unique closed symmetric monoidal structure which is uniquely determined by the requirement that the respective functor from $\S$ is symmetric monoidal. The symmetric monoidal structure in $\S_*$ and in $\Sp$ is given by the standard smash product of pointed spaces or spectra. Moreover, each of the functors above uniquely extends to a symmetric monoidal functor. %

Consider the table at the beginning of the section. \cref{thm-ggn-main} also says that an $\infty$-category satisfies the categorical property on the first column if and only if it is tensored over the $\infty$-category in the third column, following \cref{def-tensored}. \cref{prop-res} tells us that if we have an $\infty$-category tensored over one of the $\infty$-categories in (\ref{five}), then it is also tensored over any $\infty$-category which appears further to the left in the sequence, and that the action is obtained via restriction of scalars. We will now draw some consequences from this observation and from the analogous one for enrichments (\cref{prop-enrichments})

First, we set some notation:

\begin{notn} Let $\C$ be a pointed presentable $\infty$-category. By the above discussion, it is tensored over $\S_*$, so we denote the tensor by
\[-\odot-:\S_*\times \C\to \C.\]
\end{notn}
Note that since $S^0$ is the monoidal unit of $\S_*$, then $S^0\odot c \simeq c$ for any object $c$ in $\C$. The following result is a generalization of \cite[Corollary A.9]{kuhn}, see also \cref{rmk-loday}.

\begin{cor}\label{rmk-plus_otensor} Let $\C$ be a pointed presentable $\infty$-category. If $Y$ is a space, then for any object $c$ in $\C$, \[Y_+\odot c \simeq Y\otimes c\]
naturally in $Y$ and $c$.
Moreover, if $(X,x_0)$ is a pointed space, then
\begin{equation}\label{odot-cofib} X\odot c \simeq \cof(c\simeq *\otimes c \stackrel{x_0 \otimes \id}{\longrightarrow} X\otimes c).\end{equation}
\begin{proof} Only the second part needs comment. There is a cofiber sequence of pointed spaces
\[\xymatrix{S^0=\{*\}_+ \ar[r]^-{(x_0)_+} & X_+ \ar[r] & X }.\]
Applying $-\odot c$ gives the desired cofiber sequence in $\C$.
\end{proof}
\end{cor}

\begin{cor}\label{ex-spectra-odot} Let $A$ be a spectrum, $Y$ be a space and $X$ be a pointed space. Then
\[Y\otimes A \simeq \Sigma^\infty_+Y \sma A, \hspace{1cm} X\odot A \simeq \Sigma^\infty X \sma A.\]
\end{cor}
This shows that in this case $\odot$ recovers the familiar smash product of a pointed space with a spectrum, as in e.g. \cite[II.2.1]{ekmm}. 
\begin{proof} By \cref{ex-tensored-itself}, $\Sp$ is tensored over itself via the smash product, so the statements follow from \cref{prop-res}. 
\end{proof}

We can apply \cref{prop-enrichments} to presentable pointed $\infty$-categories and  presentable stable $\infty$-categories. %
We obtain that $\Omega^\infty$ of the mapping spectrum is the mapping pointed space:
\begin{cor}
 Let $\C$ be a presentable stable $\infty$-category (such as $\Sp$, for example). 
 For any two objects $A, B$ in $\C$ we have an equivalence of pointed spaces
 \[ %
 \S_*(A,B)\simeq \Omega^\infty \Sp(A,B),\]
 where $\Omega^\infty$ denotes the right adjoint to $\Sigma^\infty:\S_\ast \to \Sp$. 
\end{cor}

In the following example, we will identify $S^1\odot A$ with the more familiar bar construction on $A$, when $A$ is an augmented commutative algebra.

\begin{ex}\label{ex-ba} Let $\C$ be a presentable symmetric monoidal $\infty$-category with monoidal product $\sma$ and monoidal unit $\1$. The $\infty$-category $\CAlg(\C)_{/\1}$ of augmented commutative algebras in $\C$ is pointed, the zero object being $\1$. Let $A$ be an augmented commutative algebra in $\C$. The bar construction $BA$ of $A$ is the pushout of $\1 \leftarrow A \to \1$ \cite[5.2.2.3/4]{ha} in $\CAlg(\C)$. In fact,
\[S^1\odot A \simeq BA.\]
To see this, first write $S^1$ as the pushout of $* \leftarrow * \sqcup * \to *$ in $\S$. Tensoring with $A$ proves that $S^1\otimes A$ is the pushout of $A \leftarrow A \sma A \to A$ in $\CAlg(\C)$; the two arrows are the multiplication map of $A$. Since $S^1\odot A$ is the pushout of the unit $A\to S^1\otimes A$ along the augmentation $A\to \1$ (\ref{odot-cofib}), then $S^1\odot A$ is the pushout of $\1 \leftarrow A \sma A \to A$. %

Now consider the following diagram:
\[\xymatrix@C-.3cm@R-.3cm{
\1 & A \ar[l] \ar[r] & A \\
\1\ar[u]\ar[d] & \1 \ar[u] \ar[d] \ar[l]\ar[r] & A \ar[u]\ar[d] \\
\1 & A \ar[l]\ar[r] & A
}\]
where all the arrows are either the unit or augmentation of $A$, or identities. Taking pushouts horizontally gives the diagram $\1 \leftarrow A \to \1$ whose pushout is $BA$, and taking pushouts vertically gives the diagram $\1 \leftarrow A\sma A \to A$ whose pushout is $S^1\odot A$, as we have seen above. This proves the result, since colimits commute with colimits \cite[5.5.2.3]{htt}.

Using \cref{rmk-iterated}, since $S^n\simeq (S^1)^{\sma n}$ we deduce that $S^n\odot A$ is equivalent to the iterated bar construction $B^n A$.

As a particular example, we may take $(\C,\sma, \1)$ to be $(\Sp,\sma,\bS)$, in which case $S^1\odot A$ is $THH(A,\bS)$, the topological Hochschild homology of $A$ relative to $\bS$, and $S^n\odot -$ gives iterated $THH$ relative to $\bS$.  %
See also \cite[7.1]{kuhn}.
\end{ex}

To close the section, we will now give two remarks that compare the tensor construction to other constructions from the literature, namely, the Loday construction and the infinite symmetric product.

\begin{rmk} \label{rmk-loday}
Let $\Fin$ denote the category of finite sets. If $X:\DDelta^\op\to \Fin$ is a finite simplicial set and $R$ is an object of a presentable $\infty$-category $\C$, there is an alternative description of $X\otimes R$ whose roots go back to \cite{loday1989}\footnote{In \cref{rmk-kanfibrant} we noted that the colimit formula for tensors with spaces could be extended to arbitrary simplicial sets.}, see also \cite{pirashvili}. In \cite[Section 3]{glasman} this description is recast in an $\infty$-categorical framework as follows: %
\[X\otimes R\simeq \colim (\L_R \circ X),\] where $\L_R$, the \emph{Loday functor} for $R$, is the left Kan extension shown in the following diagram:
\[\xymatrix{& \ast \ar[r]^-{\{R\}} \ar[d]_-{\{[1]\}} & \C. \\ \DDelta^\op \ar[r]_-X & \Fin \ar[ru]_-{\L_R}}\]
Here $[1]$ denotes the one-point set. By the colimit formula for Kan extensions, $\L_R(U)$ is the coproduct of as many copies of $R$ as there are elements in $U$, for any finite set $U$, so if $\C$ has a cocartesian symmetric monoidal structure $\sma$, then $\L_R(U)=R^{\sma U}$. When $X$ is the simplicial model for the circle $S^1$ given by $\Delta^1/\partial \Delta^1$, then $S^1\otimes R$ is the cyclic bar construction on $R$, $B^\cy R$. When $\C=\CAlg(\Sp)$ this recovers the equivalence $S^1\otimes R \simeq THH(R)$ \cite{msv}.

Let $R$ be a commutative algebra in a presentable symmetric monoidal $\infty$-category $\M$ with monoidal product $\sma$. Then if $X$ is pointed, there is a version of the above construction of $X\otimes R$ relative to an $R$-module $M$, whose output is an object of $\M$ (see \cite{pirashvili} for the classical case and \cite[Section 4]{glasman} for the $\infty$-categorical version of it).

Let us look at a particular case of this: If $A$ is an augmented commutative $R$-algebra, then $R$ becomes an $A$-module via the augmentation $A\to R$, and the corresponding Loday functor of $A$ relative to $R$ is called $\L_{R,A}$: it takes a finite pointed set $U$ as input and returns the augmented commutative $R$-algebra $R\sma A^{\sma U_0}$, where $U_0$ is the set $U$ stripped of its basepoint. We then get \[R \sma_A (X\otimes_R A) \simeq \colim (\L_{R,A}\circ X).\] %
Here $X\otimes_R A$ denotes the tensor of $A$ with $X$ in the category of augmented commutative $R$-algebras; note that the forgetful functor from that category into commutative $R$-algebras is a left adjoint %
so the tensor can also be computed in the category of commutative $R$-algebras. Recalling that $R \sma_A (X\otimes_R A)$ can be constructed as the pushout of $R \leftarrow A \to X\otimes_R A$ in $\CAlg(\M)$ \cite[5.2.2.4]{ha}, then from (\ref{odot-cofib}) it follows that $R \sma_A (X\otimes_R A)\simeq X\odot_R A$, where $\odot_R$ denotes the $\odot$ construction in the pointed $\infty$-category of augmented commutative $R$-algebras. See \cite[Section 4]{kuhn} for another model of $X\odot_R A$, similar to the one in the following remark. %
\end{rmk}

\begin{rmk} \label{rmk-kuhn-sp} Kuhn \cite{kuhn} has given a different description of the tensor of an $E_\infty$-monoid in spaces %
with a pointed space. We may rephrase his construction and recover it in this context. We will do it more generally for $E_\infty$-monoids in any %
presentable $\infty$-category $\C$.

First, note that $\Mon(\C_*)\simeq \Mon(\C)$. Indeed, \begin{equation}\label{eq-monpointed}\Mon(\C_*)\simeq \C\otimes \S_*\otimes \Mon(\S)\simeq \C\otimes \Mon(\S)\simeq \Mon(\C),\end{equation} since $\Mon(\S)$ is pointed. 

Let $G\in \Mon(\C_*)$, so in particular $G$ is a functor $\Fin_*\to \C_*$. %
Let $m:\Fin_*\times \Fin_*\to \Fin_*$ denote the multiplication map which takes $(\langle n \rangle, \langle p \rangle)$ to $\langle np \rangle$. Precomposing $G$ with $m$ gives a functor $m^*G:\Fin_* \times \Fin_* \to \C_*$, whose transpose we denote by
$\overline{m^*G}:\Fin_*\to \Fun(\Fin_*,\C_*)$. Let $\bS:\Fin_*\to \S_*$ denote $\Map_{\Fin_*}(\langle1\rangle,-)$, which represents the sphere spectrum. In other words, it is the functor that considers a finite pointed set as a discrete pointed space. Taking the left Kan extension of $\overline{m^*G}$ along $\bS$ gives a functor
\[\SP(-,G):\S_*\to \Fun(\Fin_*,\C_*)\]
which preserves colimits: this notation comes from \cite{kuhn} (in the $\C=\S$ case), who explains the connection to the infinite symmetric products of \cite{mccord}. Note that $\bS(\langle 1\rangle) =S^0$ and $\overline{m^* G}(\langle 1\rangle)=G(m(\langle 1\rangle,-))=G$, so $\SP(S^0,G)=G$. Both $\SP(-,G)$ and $-\odot G$ are colimit-preserving functors with value $G$ at $S^0$; since $\S_*$ is freely generated by $S^0$ in pointed presentable $\infty$-categories (\cref{prop-freely-generated}), the two functors are equivalent. Here we used that $\Fun(\Fin_*, \C_*)$ is pointed: this follows from $\C_*$ being pointed, and is the reason we are considering $\C_*$ instead of just $\C$.

We have proven that \[\SP(-,G)\simeq -\odot G,\]recovering \cite[3.14]{kuhn} in this context (when $\C=\S$). Note that, in particular, $\SP(-,G)$ takes values in $\Mon(\C_*)$.

One can modify the construction above to obtain a similar expression for the tensor of $\Mon(\C)$ over $\Mon(\S)$, which we denote by $\odot_{\Mon}$. Indeed, a similar trick as (\ref{eq-monpointed}) shows that $\Mon(\Mon(\C))\simeq \Mon(\C)$, so a $G\in \Mon(\C)$ can be presented by a functor $G:\Fin_*\to \Mon(\C)$. Let $F:\S_*\to \Mon(\S)$ be the localization map. The statement is that the Kan extension of $G$ along $F\circ \bS$ gives $-\odot_{\Mon} G$, and the proof is similar to the one above, where we replace pointed $\infty$-categories by semiadditive ones, and we use that $F$ is symmetric monoidal, so it takes the monoidal unit of $\S_*$ to the monoidal unit of $\Mon(\S)$.

One cannot recover the tensor of $\Mon(\C)$ over $\S$ similarly as above by taking the forgetful functor $\S_*\to \S$ in place of $F$, because that forgetful functor does not preserve the monoidal unit. However, instead of considering the multiplication functor $m$, one could consider the addition functor $a:\Fin_*\times \Fin_*\to \Fin_*$ which takes $(\langle n\rangle, \langle p \rangle)$ to $\langle n+p\rangle$. In this case, if $G:\Fin_*\to \C$ is in $\Mon(\C)$, then the Kan extension of $\overline{a^*G}:\Fin_*\to \Fun(\Fin_*,\C)$ along the functor $\Fin_*\to \S$ which considers a finite pointed set as a discrete space is precisely $-\otimes G$, by a similar argument.
\end{rmk}

\section{\texorpdfstring{$E_\infty$}{E\textunderscore infty}-groups}

We now look at the case of $E_\infty$-groups $G$ more closely. We prove an $\infty$-categorical version of the splitting lemma of short exact sequences in abelian categories, and we use this to show that when $X$ is a pointed space, then the tensor $X\otimes G$ splits as a product of $G$ and $X\odot G$.

\begin{rmk} The inclusion functor $\Grp(\S) \to \Mon(\S)$ preserves colimits \cite[5.2.6.9]{ha}. %
Therefore, if $X$ is a space and $G$ is a $E_\infty$-group, then $X\otimes G$ can be computed in either $\infty$-category, and similarly for pointed $X$ and $X\odot G$. This follows either from (\ref{def-tensor}) and (\ref{odot-cofib}) or from \cref{rmk-functor-tensor}.
\end{rmk}

Let $\eqv{\Grp(\S)}{\Spc}{B^\infty}{\Omega^\infty}$ denote the $\infty$-categorical incarnation of the adjoint equivalence between $E_\infty$-groups and connective spectra \cite[5.2.6.26]{ha}. The symmetric monoidal functor $\Grp(\S)\to \Sp$ in (\ref{five}) factors as the composition of the symmetric monoidal functor $B^\infty$ followed by the inclusion $\Spc\to \Sp$ \cite[5.3(ii)]{ggn}.

\begin{prop}\label{prop-tensor_grouplike} Let $G$ be an $E_\infty$-group and $X$ be a pointed space. There is an equivalence of $E_\infty$-groups
\[X\odot G \simeq \Omega^\infty(\Sigma^\infty X\sma B^\infty G).\]
\begin{proof}
Let $F:\S_*\to \Grp(\S)$ denote the functor in (\ref{five}) and let $\boxtimes$ denote the monoidal product of $\Grp(\S)$. By \cref{prop-res}, $X\odot G\simeq FX \boxtimes G$. Applying $B^\infty$, we obtain equivalences of connective spectra
\[B^\infty(X\odot G) \simeq B^\infty(FX\boxtimes G) \simeq B^\infty FX \sma B^\infty G \simeq \Sigma^\infty X \sma B^\infty G.\]
Applying $\Omega^\infty$ to the above equivalence gives the result.
\end{proof}
\end{prop}

See also \cref{rmk-kuhn-sp} for another interpretation of $X\odot G$.

\begin{rmk} \label{rmk-odotstable} From \cref{rmk-dependence} we know that if we fix a spectrum $E$ and $X$ is a pointed space, then $X\odot E$ only depends on $\Sigma^\infty X$. In particular, if $X$ and $Y$ are pointed spaces such that for some $n>0$ we have $\Sigma^nX\simeq \Sigma^n Y$, then $X\odot E\simeq Y\odot E$. \cref{prop-tensor_grouplike} proves that we can extend this observation to $E_\infty$-groups: under the same hypotheses,  if $G$ is an $E_\infty$-group then $X\odot G \simeq Y\odot G$. This proves that the tensor of spectra or $E_\infty$-groups with pointed spaces is a stable invariant, in the sense of \cite{lindenstrauss-richter}.
\end{rmk}

The following result is a version of the splitting lemma for short exact sequences in abelian categories.

\begin{lemma}\label{lem-splitting_lemma} 
Let $\C$ be a stable $\infty$-category and let $f:U\to V$ in $\C$. If there exists a map $g:V\to U$ such that $g\circ f\simeq \id_U$, then $(g,p):V\to U\times Z$ is an equivalence, where $p:V\to Z$ is the cofiber of $f$.
\end{lemma}
\begin{proof}
	Consider the following commutative diagram
	\[\xymatrix@C-.2cm@R-.2cm{
U \ar[r]^f \ar[d]  & V \ar[d]^-p\ar[r]^g & U \ar[d] \\
\ast \ar[r] & Z\ar[r] & \ast.}\]
The left inner square is a pushout by definition, while the outer square is a pushout since $g\circ f\simeq \id_U$. It follows that the inner right square is a pushout. As $\C$ is stable, this is a pullback square as well, whence the conclusion follows.
\end{proof}

We can now give a splitting for $X\otimes G$, when $X$ is pointed:

\begin{prop}\label{cor-tensorsplit}  Let $G$ be an $E_\infty$-group and $(X,x_0)$ be a pointed space. 
There is an equivalence of $E_\infty$-groups
\[X\otimes G \simeq G\times (X\odot G)\]
which is natural in $X$ and $G$. 
Explicitly, the equivalence is given by 
\begin{equation}\label{eq-splitting}\xymatrix@C+3pc{X\otimes G \simeq X_+ \odot G \ar[r]^-{(e_+ \odot \id, p\odot \id)} & G \times (X\odot G)},
\end{equation}  where $e$ and $p$ are defined below. This equivalence makes the following diagram commute, where $\pi_0$ denotes the projection, in the domain of the top left arrow the equivalence $G\simeq \ast_+\odot G$ is understood, and in the domain of the top right arrow the equivalence $G\simeq G\times (\ast \odot G)$ is understood.
\begin{equation}\label{splitdiagram}
\xymatrix{
& G \ar[ld]_-{(x_0)_+\odot \id} \ar[rd]^-{\id \times (x_0\odot \id)} \ar[dd]|\hole^(.30)\id \\
X_+\odot G \ar[rr]^(.3)\sim \ar[rd]_-{e_+ \odot \id} && G \times (X\odot G) \ar[ld]^-{\pi_0} \\
& G
}
\end{equation}
\end{prop}
\begin{proof}
Let $e:X\to *$ denote the unique map. The split cofiber sequence of pointed spaces
	\[
	\xymatrix{ \ast_+\ar@<-.5ex>[r]_{(x_0)_+} & X_+
\ar@<-.5ex>[l]_{e_+} \ar[r]^-p & X}
	\]
induces a split cofiber sequence of $E_\infty$-groups
	\[
	\xymatrix{ G \ar[r] & X\otimes G \ar[r] & X\odot G}
	\]
	after applying the functor $-\odot G$. 
	Let $i:\Spc\to \Sp$ denote the inclusion functor, which is a left adjoint. Applying $iB^\infty$ we obtain the following cofiber sequence of spectra:
		\[
	\xymatrix{ iB^\infty G \ar[r] & iB^\infty (X\otimes G) \ar[r] & iB^\infty (X\odot G).}
	\]
	Since $\Sp$ is stable, by \Cref{lem-splitting_lemma} we have $iB^\infty (X\otimes G)\simeq  iB^\infty G\times  iB^\infty (X \odot G)$. Now note that $i$ preserves finite products. To see this, first note that the right adjoint $\tau$ of $i$ is such that $\tau \circ i \simeq \id$. Moreover, if $A$ and $B$ are connective spectra, then $i(A)\times i(B)$ is connective since homotopy groups preserve products, so $i(A)\times i(B) \simeq i(C)$ for some connective spectrum $C$. Therefore,
	\[C\simeq \tau i(C) \simeq \tau(i(A)\times i(B)) \simeq \tau i(A) \times \tau i (B) \simeq A \times B,\]
	so $i(A)\times i(B)\simeq i(A\times B)$. In particular, $iB^\infty(X\otimes G) \simeq i(B^\infty G \times B^\infty (X\odot G))$, so applying $\tau$ we obtain
	\[B^\infty(X \otimes G) \simeq B^\infty G \times B^\infty (X\odot G)\]
	in $\Spc$. Applying $\Omega^\infty$ (which preserves products) finishes the proof of the equivalence. The diagram commutes by construction.
\end{proof}

\begin{rmk} \label{rmk-xmult} When $X=S^0$, the equivalence (\ref{eq-splitting}) takes the form $G\times G \to G\times G$, $(x,y)\mapsto (xy,x)$. The diagram (\ref{splitdiagram}) becomes
\[\xymatrix{
& G \ar[ld]_-{g\mapsto (1,g)} \ar[rd]^-{g\mapsto (g,1)} \ar[dd]|\hole^(.30)\id \\
G\times G \ar[rr]^(.3)\sim \ar[rd]_-\mu && G \times G \ar[ld]^-{\pi_1} \\
& G
}\]
where $\mu$ denotes the multiplication map of $G$ and $1$ denotes the unit of $G$.
Thus, $* \otimes \id:X\otimes G \to G$ is a generalization of $\mu$. Indeed, $*\otimes \id$ is the universal map from the colimit cocone $X\otimes G$ to the identity cocone $G$. Taking $X$ to be $S^0$, this universal map is $\mu$. 
See also \cref{ex-thom-iso}.
\end{rmk}

\begin{ex} \label{ex-s1odot} Let $X=S^1$. Since the monoidal unit (namely, the point) of the symmetric monoidal $\infty$-category of $E_\infty$-monoids is a final object, then an $E_\infty$-monoid is equivalently an augmented commutative algebra in $\S$. \cref{ex-ba} then implies that $S^1\odot G\simeq BG$, so the equivalence of \cref{cor-tensorsplit} recovers the equivalence $B^\cy G \simeq G\times BG$ where $B^\cy$ denotes the cyclic bar construction (\cref{rmk-loday}).
\end{ex}

\section{Thom spectra} \label{section-thom}
Having set up the formalism of tensors and having studied the case of $E_\infty$-groups in detail, we now turn our attention to tensors of spaces with Thom spectra $Mf$. Following \cite{abghr-infty}, \cite{abg}, \cite{barthel-antolin} we define the latter in the $\infty$-categorical framework. We then determine the tensor of Thom spectra with spaces, which is particularly simple when $X$ is pointed. As a particular case, we recover the Thom isomorphism theorem of Mahowald. Finally, we look at concrete examples like suspension spectra of $E_\infty$-groups and the periodic complex cobordism spectrum $MUP$.

\subsection{Generalities} 
Let $R$ be an $E_\infty$-ring spectrum. An $R$-module $M$ is \emph{invertible} if there exists an $R$-module $N$ such that $M\sma_R N\simeq R$. We let $\Pic(R)$ be the \emph{Picard space} of $R$: this is the maximal subspace of the full subcategory of $\Mod_R$ on the invertible $R$-modules. 

\begin{defn}
Let $Z$ be a space and $f:Z\to \Pic(R)$ be a map of spaces: this is a \emph{local system} of invertible $R$-modules on $Z$. The \emph{Thom $R$-module} of $f$ is defined as
\[Mf\coloneqq \colim(\xymatrix{Z\ar[r]^-f & \Pic(R) \ar@{^(->}[r] & \Mod_R}).\]
Note that this defines a functor $M:\S_{/\Pic(R)} \to \Mod_R$.
\end{defn}

\begin{ex} Let $f:Z\to \Pic(R)$ be the constant map at $R$. Then $Mf\simeq R\sma \Sigma^\infty_+Z$. %
Indeed, $Mf$ is the colimit of $\xymatrix{Z\ar[r]^-{\{\bS\}} & \Sp \ar[r]^-{R\sma -} & \Mod_R}$, but $R\sma -$ preserves colimits, so $Mf\simeq R\sma (Z\otimes \bS) \simeq R\sma \Sigma^\infty_+(Z)$ using \cref{ex-spectra-odot}. %
\end{ex}

As noted in \cite[7.7]{abg}, \cite[Section 3]{barthel-antolin}, $\Pic(R)$ is an $E_\infty$-group. If $G$ is an $E_\infty$-monoid and $f$ is an $E_\infty$-map, then $Mf$ gets an $E_\infty$-structure:

\begin{prop} \cite[8.1]{abg}, \cite[3.2]{barthel-antolin} \label{prop-thomalgebra} If $G$ is an $E_\infty$-monoid and $f:G\to \Pic(R)$ is an $E_\infty$-map, then there is an $E_\infty$-$R$-algebra $Mf$ such that its underlying $R$-module is given by the above colimit.
\end{prop}

\begin{notn}Let $\bS[-]:\Mon(\S)\to \CAlg(\Sp)$ denote the functor induced on commutative algebras by the symmetric monoidal functor $\Sigma^\infty_+:\S\to\Sp$. For $R$ an $E_\infty$-ring spectrum, let $R[-]:\Mon(\S)\to \CAlg_R$ denote the functor $R\sma \bS[-]$.
\end{notn}

\begin{ex} \label{ex-trivialmult} If $f:G\to \Pic(R)$ is the $E_\infty$-map constant at $R$, then $Mf\simeq R[G]$ as $E_\infty$-$R$-algebras. This follows e.g. from \cite[3.16]{barthel-antolin} (see also \cite[8.4]{abg}), since $f$ admits an obvious $E_\infty$ $R$-orientation. %
\end{ex}

The proposition above is a consequence of the following theorem. First, recall that if $\C$ is a symmetric monoidal $\infty$-category and $c\in \CAlg(\C)$, then the over-category $\C_{/c}$ is also symmetric monoidal \cite[2.2.2.4]{ha}: if $f:x\to c$ and $g:y\to c$, then their monoidal product is given by the composition
\[\xymatrix{x\sma y \ar[r]^-{f\sma g} & c \sma c \ar[r]^-\mu & c}\]
where $\mu$ is the multiplication map of $c$. 
A commutative algebra in $\C_{/c}$ is then a commutative algebra map $c'\to c$ from a commutative algebra $c'$ in $\C$.

\begin{thm}\cite[8.1]{abg} The Thom $R$-module functor %
\[M: \S_{/\Pic(R)} \to \Mod_R\]
is symmetric monoidal.
\end{thm}

Since symmetric monoidal functors preserve commutative algebras, we get a functor
\begin{equation}\label{M-mult} M:\Mon(\S)_{/\Pic(R)} \to \CAlg_R,\end{equation}
thus recovering \Cref{prop-thomalgebra}. The functor (\ref{M-mult}) is also symmetric monoidal \cite[3.2.4.3]{ha}. %
In particular, this says that if $f:G\to \Pic(R)$ and $g:H\to \Pic(R)$ are two $E_\infty$-maps, then
\begin{equation}\label{eq-thommonoidal} M(\xymatrix{G\times H \ar[r] ^-{f\times g} & \Pic(R)\times \Pic(R) \ar[r]^-\mu & \Pic(R)}) \simeq Mf \sma_R Mg
\end{equation}
as $E_\infty$-$R$-algebras.\\

Let $A$ be an $E_\infty$-$R$-algebra. Let us define the $E_\infty$-monoid $\Pic(R)_{\downarrow A}$ as the following pullback of symmetric monoidal $\infty$-categories:
\[\xymatrix{
\Pic(R)_{\downarrow A} \ar[r] \ar[d] & (\Mod_R)_{/A} \ar[d] \\
\Pic(R) \ar[r] & \Mod_R.}\]
Antolín-Camarena and Barthel have given the following universal property for the $E_\infty$-structure on $Mf$:

\begin{thm}\cite[3.5]{barthel-antolin} \label{prop-univthom}
Let $G$ be an $E_\infty$-monoid and $f:G\to \Pic(R)$ be an $E_\infty$-map. The $\ER$-algebra structure on $Mf$ is characterized by the following universal property: the space of $\ER$-algebra maps $\Map_{\CAlg_R}(Mf,A)$ is equivalent to the space $\Map_{\Mon(\S)/_{\Pic( R)}}(G,\Pic(R)_{\downarrow A})$ of $E_\infty$-lifts of $f$ indicated below:
\[\xymatrix{
& \Pic(R)_{\downarrow A} \ar[d] \\
G \ar@{.>}[ru] \ar[r]_-f & \Pic(R).}\]
\end{thm}

\subsection{Behavior with respect to the tensor}

Let $R$ be an $E_\infty$-ring spectrum. We will denote the tensor of $\CAlg_R$ over spaces by $\otimes_R$. Recall that $S^1\otimes_R -$ is $THH^R$, topological Hochschild homology of $E_\infty$-$R$-algebras over the base $R$ \cite{msv}. Let $f:G\to \Pic(R)$ be a map of $E_\infty$-monoids. Let $X$ be a space. We now prove that $X \otimes_R - $ and the Thom spectrum functor commute:

\begin{prop} \label{prop-tensorcommutes} There is an equivalence of $E_\infty$-$R$-algebras
\[X\otimes_R Mf \simeq M\Big(X\otimes G \stackrel{\ast \otimes \id}{\longrightarrow} \ast \otimes G \simeq G \stackrel{f}{\to}  \Pic(R)\Big). \]
\begin{proof} Let $A$ be an $E_\infty$-$R$-algebra. Using the universal property of $\otimes_R$, \Cref{prop-univthom} and \Cref{lemma-over} applied to $\C=\Mon(\S)$, $K=\Pic(R)$ and $H=\Pic(R)_{\downarrow A}$, we obtain:
\begin{align*}
\Map_{\CAlg_R}(X\otimes_R Mf,A)
&\simeq \Map_\S(X,\Map_{\CAlg_R}(Mf,A)) \\
&\simeq \Map_\S(X,\Map_{\Mon(\S)_{/\Pic(R)}}(G,\Pic(R)_{\downarrow A})) \\
&\simeq \Map_{\Mon(\S)_{/\Pic(R)}}(X\otimes G,\Pic(R)_{\downarrow A}) \\
&\simeq \Map_{\CAlg_R}\Big(M\Big(X\otimes G \stackrel{\ast \otimes \id}{\longrightarrow} \ast \otimes G \simeq G \stackrel{f}{\to}  \Pic(R)\Big),A\Big).\qedhere
\end{align*}
\end{proof}
\end{prop}

\begin{rmk} \label{rmk-newproof}We have more recently found an alternative proof of the above result. Let us sketch it. First, note that the symmetric monoidal functor $M:\S_{/\Pic(R)}\to \Mod_R$ preserves colimits \cite[8.1]{abg}, so the induced functor $M:\Mon(\S)_{/\Pic(R)} \to \CAlg_R$ preserves colimits as well. Second, remark that $X\otimes G \xrightarrow{*\otimes \id}G \xrightarrow{f} \Pic(R)$ is actually the tensor of $X\in \S$ with $(G\xrightarrow{f}\Pic(R))\in \Mon(\S)_{/\Pic(R)}$ by \cref{lemma-over}. Since tensors are colimits, the proposition follows.
\end{rmk}

\begin{rmk} \label{rmk-mfunit} If in the proposition above we choose a basepoint $x_0$ in $X$, then $X\otimes_R Mf$ gets the structure of an $E_\infty$-$Mf$-algebra via the map $x_0\otimes \id: Mf\simeq *\otimes_R Mf \to X\otimes_R Mf$. Similarly, $M(X\otimes G \stackrel{\ast \otimes \id}{\longrightarrow} \ast \otimes G \simeq G \stackrel{f}{\to}  \Pic(R))$ gets the structure of an $E_\infty$-$Mf$-algebra by Thomifying the map $x_0\otimes \id:G\to X\otimes G$. Since the equivalence of \cref{prop-tensorcommutes} is natural in $X$, it commutes with these unit maps, so in fact it is an equivalence of $E_\infty$-$Mf$-algebras once we fix a basepoint in $X$.
\end{rmk}

\begin{thm}\label{thm-tensorthom} 
Suppose $G$ is an $E_\infty$-group and $X$ is pointed. There is an equivalence of $E_\infty$-$R$-algebras
\[X\otimes_R Mf \simeq Mf \sma \bS[X\odot G].\]
\end{thm}
\begin{proof} By \Cref{prop-tensorcommutes} and \Cref{cor-tensorsplit}, we have equivalences of $E_\infty$-$R$-algebras
\begin{align}
\nonumber X\otimes_R Mf
& \simeq M\Big(\xymatrix{X\otimes G \ar[r]^-{\ast \otimes \id} & \ast \otimes G \simeq G  \ar[r]^-f &  \Pic(R)}\Big) \\
& \simeq M\Big( \xymatrix{G\times (X\odot G) \ar[r]^-{\pi_1} & G \ar[r]^-f & \Pic(R) }\Big) \label{eq-tensorcommutes-1}\end{align}
where in the second equivalence we have applied functoriality of $M$. 
Since $R$ is the unit of $\Pic(R)$, the following diagram commutes, %
\[\xymatrix{
G\times (X\odot G) \ar[r]^-{\pi_1\simeq(\id\times \ast)} \ar[rd]_-{f\times \{R\}}
& G\times \ast \simeq G \ar[d]^-{f\times \{R\}} \ar[r]^-f
& \Pic(R) \\ 
& \Pic(R) \times \Pic(R) \ar[ru]_-\mu
}\]
so (\ref{eq-tensorcommutes-1}) is equivalent to Thom $E_\infty$-$R$-algebra of $\mu\circ (f\times \{R\})$. By monoidality of $M$ (\ref{eq-thommonoidal}), this is equivalent to \[Mf\sma_R M\Big(X\odot G \stackrel{\{R\}}{\to} \Pic(R)\Big).\]
By \Cref{ex-trivialmult}, it is equivalent to $Mf\sma_R R [X\odot G]\simeq Mf \sma \bS[X\odot G]$. %
\end{proof}

\begin{rmk} \label{rmk-mfunit2}
In \cref{thm-tensorthom}, we can consider $X\otimes_R Mf$ as an $E_\infty$-$Mf$-algebra as in \cref{rmk-mfunit}, and we can consider $Mf\sma \bS[X\odot G]$ as an $E_\infty$-$Mf$-algebra via the coproduct inclusion into the first factor. In this way, the equivalence $X\otimes_R Mf \simeq Mf \sma \bS[X\odot G]$ commutes with the unit maps from $Mf$, so is in fact an equivalence of $E_\infty$-$Mf$-algebras. The gist of it is that in (\ref{splitdiagram}) the maps from $G$ commute with the equivalence.
\end{rmk}

\begin{rmk} \label{rmk-before} Taking \cref{prop-tensor_grouplike} into account, \cref{thm-tensorthom} is similar to \cite[1.1]{schlichtkrull-higher} in an $\infty$-categorical setting. %
Note that whereas we consider Thom spectra of maps into $\Pic(R)$ for any $E_\infty$-ring spectrum $R$, Schlichtkrull's result is for Thom spectra of maps into $BGL_1(\bS)$. Recall that the \emph{space} $\Pic(R)$ is equivalent to the product of $BGL_1(R)$ with $\pi_0(\Pic(R))$; %
the latter is the classical Picard group of the homotopy category of $\Mod_R$, i.e. the group of isomorphism classes of $R$-modules $M$ such that there exists an $R$-module $M'$ satisfying that $M\sma_R M'$ is isomorphic to $R$ in the homotopy category of $R$-modules.

Considering Thom spectra of maps into $\Pic(R)$ as we do is therefore more general (for example, they may be non-connective whereas Thom spectra of maps into $BGL_1(\bS)$ are connective), and already taking maps into $\Pic(\bS)\simeq \Z \times BGL_1(\bS)$ allows for a nice application, see \cref{ex-mup}. %

Note that the result of Schlichtkrull has also been generalized before in a different direction. In \cite[4.2]{klang-thom}, the author determined the factorization homology of Thom $R$-algebras. On one hand, it is more general, as it can be applied to Thom $E_n$-$R$-algebras instead of only $E_\infty$; on the other hand, she only considers Thom spectra of maps into $BGL_1(R)$, whereas, more generally, we consider maps into $\Pic(R)$. Moreover, factorization homology takes a manifold with extra structure as an input, whereas we consider tensors with completely general spaces $X$. The setting and the techniques are very different, and as a consequence, the expression of the result (of the factorization homology of structured manifolds with coefficients in Thom $E_n$-$R$-algebras in her case and of tensors of Thom $E_\infty$-$R$-algebras with spaces in our case) takes a very different form at a first glance.
\end{rmk}

\begin{rmk} \label{rmk-otimesstable}In \cref{rmk-odotstable} we observed that $-\odot G$ is a stable invariant, i.e. if $X$ and $Y$ are pointed spaces such that $\Sigma^nX\simeq \Sigma^nY$ for some $n>0$, then $X\odot G \simeq Y\odot G$. From \cref{thm-tensorthom} we deduce that $-\otimes_R Mf$ is a stable invariant as well. This was observed in \cite[5.1]{lindenstrauss-richter}, but since they use Schlichtkrull's result from \cite{schlichtkrull-higher}, they only consider maps to $BGL_1(\bS)$ as input for Thom spectra, instead of a more general $\Pic(R)$ for an $E_\infty$-ring spectrum $R$.
\end{rmk}

\begin{ex} \label{ex-thom-suspension} Let $G$ be an $E_\infty$-group and let $f:G\to \Pic(R)$ be the constant map at $R$.  By \cref{ex-trivialmult}, $Mf\simeq R[G]$. From \cref{thm-tensorthom} we deduce that
\[X\otimes_R R[G] \simeq R[G] \sma \bS[X\odot G].\]
\end{ex}

\begin{ex} Let $X=S^1$. Let $G$ be an $E_\infty$-group and let $f:G\to \Pic(R)$ be an $E_\infty$-map. Since $S^1\odot G\simeq BG$ (\cref{ex-s1odot}), then the formula of \cref{thm-tensorthom} amounts to an equivalence \begin{equation}\label{thh-mf} THH^R(Mf)\simeq Mf \sma \bS[BG]\end{equation} of $E_\infty$-$R$-algebras. The equivalence (\ref{thh-mf}) has antecedents in the literature: besides the papers of Schlichtkrull and Klang mentioned in \cref{rmk-before}, there is \cite{blumberg} for Thom $E_\infty$-ring spectra of $E_\infty$-maps into $BGL_1(\bS)$, \cite{bcs} for Thom spectra of $E_3$-maps into $BGL_1(\bS)$, and \cite{bss-generalizedthom} for Thom $E_\infty$-$R$-algebras of $E_\infty$-maps into $BGL_1(R)$. These three papers take place in model-categorical contexts.
\end{ex}

\begin{ex} \label{ex-thom-iso} Consider a pointed space of the form $X_+=*\sqcup X$ where $X$ is a space. Then \cref{thm-tensorthom} amounts to an equivalence of $E_\infty$-$R$-algebras 
\begin{equation}\label{ex-thom-general}Mf\sma_R (X\otimes_R Mf) \simeq Mf \sma \bS[X\otimes G],\end{equation}
essentially given by Thomifying the map $X\otimes G \to G\times (X\odot G)$ of (\ref{eq-splitting}). When $X=\{*\}$, the equivalence becomes
\begin{equation}\label{ex-thom}Mf \sma_R Mf \simeq Mf \sma \bS[G],\end{equation}
which is the Thom isomorphism theorem going back to \cite{mahowald-thom}, see also \cite[3.16/17]{barthel-antolin}. Indeed, in this case, the equivalence is given by Thomifying the map $G\times G \to G\times G$, $(x,y)\mapsto (xy,y)$ (see \cref{rmk-xmult}). Thus, we can think of the equivalence (\ref{ex-thom-general}) for general $X$ as a generalization of this Thom isomorphism.
\end{ex}

\begin{ex} \label{ex-mup}Consider the stable $J$-homomorphism $BU\times \Z\to \Pic(\bS)$, an $E_\infty$-map of $E_\infty$-groups. Its Thom $E_\infty$-ring spectrum is the periodic complex cobordism $MUP$. Note that $MUP\simeq \bigvee_{n\in \Z} \Sigma^{2n} MU$ as spectra. %
\cref{thm-tensorthom} gives an equivalence of $E_\infty$-ring spectra
\[X\otimes MUP\simeq MUP \sma \bS[X\odot (BU \times \Z)]\]
for all pointed spaces $X$. For example, for $X=S^1$, since $B(BU\times \Z)\simeq U$ this gives an equivalence of $E_\infty$-ring spectra 
\[THH(MUP)\simeq MUP \sma \bS[U].\] %
This equivalence was briefly mentioned in \cite[8.6]{sagave-schlichtkrull}. See also \cref{ex-mup2}.
\end{ex}

\begin{ex} Let $H\Z P$ denote the periodic integral homology spectrum, which is the Thom $E_\infty$-$H\Z$-algebra of the $E_\infty$-map $\Z\to \Pic(H\Z)$ that sends $n$ to $\Sigma^{2n}H\Z$ \cite[2.3]{hahn-yuan}. Note that $H\Z P\simeq \bigvee_{n\in \Z} \Sigma^{2n} H\Z$ as $H\Z$-modules. \cref{thm-tensorthom} gives an equivalence of $E_\infty$-$H\Z$-algebras
\[X\otimes_{H\Z}H\Z P \simeq H\Z P \sma \bS[X\odot \Z].\]
For example, for $X=S^1$, we get $THH^{H\Z}(H\Z P)\simeq H\Z P \sma \bS[S^1]$.
\end{ex}

\begin{ex} Let $KU$ denote the periodic complex topological $K$-theory $E_\infty$-ring spectrum. We will see in \cref{ex-s0} that $S^1\otimes KU\simeq KU \sma \bS[BK(\Z,2)]$. Formula (\ref{thh-mf}) suggests that $KU$ could be the Thom spectrum of an $E_\infty$-map $K(\Z,2)\to \Pic(\bS)$. However, this is not the case. For if it were, then by the Thom isomorphism theorem (\cref{ex-thom-iso}) we would have that $KU\sma KU$ is equivalent to $KU \sma K(\Z,2)_+$, which is not the case, as recalled in \cref{ex-s0}. In fact, $KU$ is not even the Thom spectrum of an $E_1$-map $K(\Z,2)\to \Pic(\bS)$, since the Thom isomorphism holds for $E_1$-maps. %
Note that in \cite{ahl-ku-thom} the authors prove that the connective complex $K$-theory $ku$ is not the Thom spectrum of any $E_3$-map $X\to BGL_1(\bS)$ where $X$ is any grouplike $E_3$-space.
\end{ex}

\section{$X$-base change} \label{section-basechange}

If $A\to B$ is a morphism of $E_\infty$-ring spectra, for any space $X$ there is an induced map $X\otimes A \to X\otimes B$. Sometimes, knowing $X\otimes A$ one can get to $X\otimes B$: the  Weibel-Geller theorem \cite{weibel-geller}, for example, asserts that if $A\to B$ is an étale extension of commutative rings, then $HH(B)$, the Hochschild homology of $B$, can be computed as $HH(A)\otimes_A B$; the topological analog of this theorem was proven by Mathew \cite{mathew-thh}. We will now generalize this question to arbitrary tensors: see \cref{def-xbase} where we introduce the notion of $X$-base change. We will prove that $S^n$-base change is enough to guarantee $X$-base change for any $(n-1)$-connected pointed $X$. In this section, we work in an arbitrary presentable $\infty$-category; we will specialize to $E_\infty$-ring spectra in \cref{section-base-rings}.

\begin{defn} \label{def-xbase} Let $X$ be a pointed space and $\C$ be a presentable $\infty$-category. We say that a map $c\to d$ in $\C$ \emph{satisfies $X$-base change} if the diagram in $\C$
\[\xymatrix@C-0.2cm@R-.2cm{c\ar[r] \ar[d] & d \ar[d] \\ X\otimes c \ar[r] & X\otimes d}\]
is a pushout, where the vertical maps are given by the inclusion of the basepoint in $X$. Equivalently, we are asking for the pushout map
\begin{equation}\label{eq-xbase}(X\otimes c) \sqcup_c d \to X\otimes d\end{equation}
to be an equivalence. By Yoneda and the tensor -- mapping space adjunction, this is equivalent to
\begin{equation}\label{eq-xbase-yon}\xymatrix@C-.2cm@R-.2cm{\Map_\S(X,\Map_\C(d,z)) \ar[r] \ar[d] & \Map_\C(d,z) \ar[d] \\ \Map_\S(X,\Map_\C(c,z)) \ar[r] & \Map_\C(c,z)}\end{equation}
being a pullback in $\S$ for all $z\in \C$, where the horizontal maps are the evaluation maps at the basepoint of $X$.
\end{defn}

\begin{ex} Any map $c\to d$ satisfies $X$-base change when $X$ is contractible.
\end{ex}

We will now prove that the base change property is closed under many operations; see also \cref{rmk-notclosed} for %
a negative result. That remark and the following proposition settle the question of the closure of the base change property under limits and colimits in general, except possibly for infinite products.

\begin{prop} \label{lemma-basechange}  Let $f:c\to d$ be a map in a presentable $\infty$-category $\C$. 
\begin{enumerate}
\item \label{item-products} %
Let $X$ and $Y$ be pointed spaces. Suppose that $f$ satisfies $X$-base change and $Y$-base change. Then it satisfies $X\times Y$-base change.
\item  \label{item-colim} Let $F:I\to \S_*$ be a diagram of pointed spaces. %
Suppose that $f$ satisfies $F(i)$-base change for all $i$. Then it satisfies $(\colim_{\S_*} F)$-base change. 
\item Let $X$ and $Y$ be pointed spaces. Suppose that $f$ satisfies $X$-base change and $Y$-base change. Then it satisfies $X\sma Y$-base change.
\end{enumerate}

\begin{proof} 
\begin{enumerate}

\item
Let $z\in \C$. Applying $\Map_\S(Y,-)$ to the pullback diagram (\ref{eq-xbase-yon}) and using the product -- mapping space adjunction in $\S$ we get a pullback diagram
\[\xymatrix@C-.2cm@R-.2cm{\Map_\S(X\times Y, \Map_\C(d,z)) \ar[r]\ar[d] & \Map_\S(Y,\Map_\C(d,z)) \ar[d] \\ \Map_\S(X\times Y,\Map_\C(c,z)) \ar[r] & \Map_\S(Y,\Map_\C(c,z)).}\]
Pasting this pullback diagram with the pullback diagram (\ref{eq-xbase-yon}) where $X$ has been replaced with $Y$ gives us the result.

\item %
Let $U:\S_*\to \S$ denote the forgetful functor. Let $I^\Kan\in \S$ denote the colimit of the constant diagram $\{*\}:I\to \S$ (cf. \cref{rmk-kanfibrant}). Taking the colimit of the pushout squares witnessing that $f$ satisfies $F(i)$-base change for all $i$, we get a pushout square in $\C$
\[\xymatrix{
I^\Kan\otimes c \ar[r] \ar[d] & I^\Kan\otimes d \ar[d] \\ \colim_\S(UF) \otimes c \ar[r] & \colim_\S(UF)\otimes d 
}\]
since colimits commute with pushouts and tensors. Now, recall that $U(\colim_{\S_*}F)$ is computed by the pushout square in $\S$
\[\xymatrix{
I^\Kan \ar[r] \ar[d] & \ast \ar[d] \\ \colim_\S(UF) \ar[r] & U(\colim_{\S_*}(F)).
}\]
The above recipe for colimits in $\S_*$ is a classical result, but we quickly sketch an $\infty$-categorical proof. The universal property of pushouts implies that the inclusion $\S_*\to \S^{\Delta^1}$ from pointed spaces to arrows of spaces has a left adjoint $L$ that takes $X\to Y$ to its pushout along $X\to *$. Colimits in $\S^{\Delta^1}$ are computed pointwise, and colimits in $\S_*$ are computed by applying $L$ to their inclusion into $\S^{\Delta^1}$ \cite[5.2.7.5]{htt}. This gives the result. 

Applying $-\otimes f$ to that square we get the following commutative cube
\[\xymatrix{
I^\Kan \otimes c \ar[rr] \ar[dd] \ar[rd] && I^\Kan \otimes d \ar[dd]|!{[d]}\hole \ar[rd] \\
& c \ar[rr] \ar[dd] && d \ar[dd] \\
\colim_\S(UF) \otimes c \ar[rr]|!{[r]}\hole \ar[rd] && \colim_\S(UF) \otimes d \ar[rd] \\
& U(\colim_{\S_*}F) \otimes c \ar[rr] && U(\colim_{\S_*}F)\otimes d.
}\]
The back face and the left and right sides are pushouts, hence the front side is a pushout square as well by the pasting law \cite[4.4.2.1]{htt}, proving the result. %

\item %
Note that $X\sma Y$ is the cofiber of $X\vee Y \to X\times Y$, so the result follows from the previous points.\qedhere
\end{enumerate}
\end{proof}
\end{prop}
\begin{rmk} \label{rmk-notclosed}
Base change is not stable under pullbacks in general. Suppose it were, and let $f$ be a map that satisfies $X$-base change for some given connected pointed space $X$. Then $f$ would satisfy $\Omega X = *\times_X*$-base change. In \cref{ex-s0} we will give an example of an $f$ that satisfies $X$-base change for all connected pointed $X$ but does not satisfy $S^0$-base change. Now take $X$ to be $\mathbb{R}P^\infty=B\Z/2$: we would have that $f$ satisfies $\Omega B\Z/2 \simeq S^0$-base change, getting a contradiction.
\end{rmk}

The following theorem is an application of the previous proposition:

\begin{thm} \label{thm-s1enough}  Let $f:c\to d$ be a map in a presentable $\infty$-category. Let $n\geq 0$. Suppose $f$ satisfies $S^n$-base change. Then $f$ satisfies $X$-base change for any $(n-1)$-connected pointed space $X$.\footnote{Recall that all pointed spaces are $(-1)$-connected, which just means non-empty.}
\end{thm}
\begin{proof}Let $X$ be a pointed space: it is the sequential colimit in $\S$ of its skeleta $X_i\to X_{i+1}$. Suppose $X$ is $(n-1)$-connected. It suffices to prove that $f$ satisfies $X_i$-base change for all $i$.

For $0\leq i \leq n-1$ (there are no such $i$ if $n=0$ and in this case this step is skipped), since $X$ is $(n-1)$-connected, we can assume that $X_i$ is a point, so $f$ satisfies $X_i$-base change for these values.

For $i=n$, using the above assumption we have that $X_n$ is a wedge of copies of $S^n$, so $f$ satisfies $X_n$-base change.

For $i\geq n$ we do induction. We have that $X_{i+1}$ is the pushout of $\bigvee * \leftarrow \bigvee S^i \to X_i$. It now suffices to observe that $f$ satisfies $S^i$-base change. Indeed, this follows by induction on $i\geq n$, by noting that $S^{i+1}$ is the pushout of $*\leftarrow S^i \to *$.
\end{proof}

See \cref{thm-etale-change} and \cref{cor-galois-change} for concrete applications of \cref{thm-s1enough}.

\section{The Thom isomorphism is not $S^0$-base change} \label{section-thomisonots0}

A map $g:c\to d$ satisfies $S^0$-base change if and only if $g\sqcup \id:c\sqcup d \to d \sqcup d$ is an equivalence. 
Note that this is a stronger condition than $c\sqcup d \simeq d \sqcup d$: for an example in %
the category of sets, 
no function $g:\{0\}\to \N$ is such that $g\sqcup \id$ is an equivalence, but $\{0\}\sqcup \N \simeq \N \sqcup \N$. 

\par 

 In this section, we show that the Thom isomorphism is another example of this phenomenon. %
 Let $G$ be an $E_\infty$-group, let $R$ be an $E_\infty$-ring spectrum and let $f:G\to \Pic(R)$ be an $E_\infty$-map. 
 The Thom isomorphism theorem 
 of \cref{ex-thom-iso} gives us an equivalence $R[G] \sma_R Mf \simeq Mf \sma_R Mf$,
 which suggests the possibility that there exists a map $R[G] \to Mf$ of $E_\infty$-$R$-algebras 
 that satisfies $S^0$-base change. We will show that this is not possible for maps of the form
 $M(h): R[G] \to Mf$, under a certain naturality assumption for the maps $h:G\to G$, and with the hypotheses that $f$ is non-torsion and that the multiplication of $Mf$ is not an equivalence (this holds for example for $Mf=MUP$, see \cref{ex-picmup}).

The gist of the proof is \cref{cor-natcol-omega}. It is a statement about $E_\infty$-groups: it says that any natural collection of endomorphisms over a fixed group $P$ (in the sense of \cref{def-naturalcoll}; we will then take $P=\Pic(R)$) has to be of the form $g\mapsto g^n$ for a fixed integer $n$.
The statement about the Thom isomorphism not being $S^0$-base change, \cref{thm-thomisnots0}, follows from general properties of Thom spectra.
 
The proof of \cref{thm-thomisnots0} is similar in spirit to the proof of the Thom isomorphism of \cref{ex-thom-iso}, in the sense that both of them reduce to a statement about $E_\infty$-groups: in the Thom isomorphism case, the key observation was that the shear map $G\times G \to G \times G$, $(x,y)\mapsto (xy,y)$ is an equivalence.

\begin{defn} \label{def-naturalcoll}Let $P$ be an object in an $\infty$-category $\C$. Let $\pi:\CP\to \C$ be the projection functor. A \emph{natural collection of endomorphisms over $P$} is a natural transformation $\pi\Rightarrow \pi$. We assemble these into the space
\[\NatCol_P(\C)\coloneqq \Map_{\Fun(\C_{/P},\C)}(\pi,\pi).\]
\end{defn}

\begin{rmk}If $H$ is a natural collection of endomorphisms over $P$, then $H$ takes an arrow $f:G\to P$ and associates to it an endomorphism $H(f):G\to G$; moreover, if $g:G'\to G$, then naturality of $H$ means that the following diagram commutes.
\[\xymatrix{G' \ar[r]^-g \ar[d]_-{H(f\circ g)} & G \ar[d]^-{H(f)} \\
G' \ar[r]_-g & G
}\]
\end{rmk}

Let us give a class of examples of natural collections of endomorphisms over an $E_\infty$-group $P$. We first need a definition and a lemma.%

\begin{defn} 
	Let $G$ be an $E_\infty$-group and $n \in \mathbb{Z}$. We define a map of $E_\infty$-groups
	$$\mathbb{P}_n: G \to G$$
	by 
	$$
	\mathbb{P}_n = 
	\begin{cases}
		G \xrightarrow{ \ \Delta \ } G^n \xrightarrow{ \ \mu \ } G & \textup{if } n \geq 0 \\
		G \xrightarrow{ \ \Delta \ } G^{-n} \xrightarrow{ \ \mu \ } G \xrightarrow{ \ i \ } G & \textup{if } n < 0 
	\end{cases}
	$$
	where $\mu$ is the multiplication map and $i$ is the inverse map.
\end{defn}

\begin{lemma}\label{rmk-pn} 
Let $n\in \Z$. There exists a natural transformation
\[\begin{tikzcd}
	{\Grp(\S)} && {\Grp(\S)}
	\arrow["{\id}"{name=0}, from=1-1, to=1-3, curve={height=-12pt}]
	\arrow["{\id}"{name=1, swap}, from=1-1, to=1-3, curve={height=12pt}]
	\arrow[Rightarrow, "{\mathbb{P}_n}", from=0, to=1, shorten <=2pt, shorten >=2pt]
\end{tikzcd}\]
 with components given by the $\bP_n$ above. Moreover, $\bP_n\simeq \bP_m$ if and only if $m=n$.

\begin{proof} Let $n\in \Z$. %
	Since $\Delta$ is a natural transformation, it suffices to observe that the multiplication map and the inverse map are natural transformations. For the multiplication, take $n\geq 0$, and observe there is a functor
	\[\Fin_*\to \Fun(\Mon(\S),\Mon(\S))\]
	that takes the active map $\langle n\rangle \to \langle 1 \rangle$ to the natural transformation $\mu:(-)^n\Rightarrow \id_{\Mon(\S)}$. Indeed, by adjunction such a functor corresponds to a functor $\Fin_*\times \Mon(\S)\to \Mon(\S)$. Replacing the first $\Mon(\S)$ by the equivalent $\Mon(\Mon(\S))$ \cite[3.2.4.5/2.4.2.5]{ha}, this functor is precisely the evaluation functor (recall that $\Mon(\Mon(\S))$ is the full subcategory of special $\Gamma$-objects in $\Fun(\Fin_*,\Mon(\S))$.
	
	To see that inversion constitutes a natural transformation $\id_{\Grp(\S)}\Rightarrow \id_{\Grp(\S)}$, consider the shear map $s:G\times G \to G\times G$. It is defined as the projection $\pi_1$ on the first factor and multiplication on the second. We have just seen that multiplication is natural: the projection is analogously proven to be natural, by considering the corresponding inert map $\langle 2\rangle \to \langle 1\rangle$. Therefore, the shear map is the component of a natural transformation. Since $G$ is an $E_\infty$-group, the shear map is invertible, thus the inverse $s^{-1}$ is the component of a natural transformation. Now, the inversion $i$ is the composition $G\simeq G \times \ast \xrightarrow{\id\times \mu_0} G\times G \xrightarrow{s^{-1}} G\times G \xrightarrow{\pi_2} G$: all of these are components of natural transformations.
	
Finally, one quickly proves that $\bP_n\simeq \bP_m$ if and only if $m=n$ by considering cyclic groups.
\end{proof}
\end{lemma}

Now, if we fix an $E_\infty$-group $P$, we can whisker the natural transformation $\bP_n:\id\Rightarrow \id$ with $\pi:\Grp(\S)_{/P}\to \Grp(\S)$ to get a natural collection of endomorphisms over $P$, which we also denote by $\bP_n$. 
Our goal now is to show that these $\mathbb{P}_n$ are in fact the only examples of natural collections of endomorphisms over $P$.  We will prove that in \cref{cor-natcol-omega}, but first we need a couple of results.\\

Suppose $\C$ is a pointed presentable $\infty$-category. If $0$ is a zero object of $\C$, we have a trivial Kan fibration $\C_{/0}\to \C$, so we can choose a section $\C\to \C_{/0}$ of it. Composing it with the functor $\C_{/0}\to \C_{/P}$ induced by the zero map $0\to P$, we get a functor $Z:\C\to \C_{/P}$ which functorially chooses zero maps over $P$.

We have two functors
\[\adj{\Fun(\C,\C)}{\Fun(\C_{/P},\C)}{\pi^*}{Z^*}.\]
Note that $\pi\circ Z \simeq \id$, so $Z^* \circ \pi^* \simeq \id$. Taking appropriate mapping spaces, we get two maps
\begin{equation}\label{eq-natnatcol}\adj{\Nat(\id_\C,\id_\C)}{\NatCol_P(\C)}{\pi^*}{Z^*}\end{equation}
showing $\Nat(\id_\C,\id_\C)$ as a retract of $\NatCol_P(\C)$. Here $\Nat(\id_\C,\id_\C)$ denotes the mapping space $\Map_{\Fun(\C,\C)}(\id_\C,\id_\C)$. We will now show that under an additional condition of semiadditivity, the maps in (\ref{eq-natnatcol}) are inverse homotopy equivalences.

\begin{prop} \label{prop-nat-natcol} If $\C$ is a semiadditive presentable $\infty$-category, then
\[\adj{\Nat(\id_\C,\id_\C)}{\NatCol_P(\C)}{\pi^*}{Z^*}\]
are inverse homotopy equivalences, for all $P\in \C$.
\begin{proof} 
We already observed that $Z^* \circ \pi^*\simeq \id$,  so $\pi^*$ is an inclusion of path-components (i.e. a fully faithful functor of $\infty$-groupoids). We only need to prove it is also a surjection on $\pi_0$ (i.e. an essentially surjective functor of $\infty$-groupoids).
 More concretely: Let $H\in \NatCol_P(\C)$. In particular, $H:\CP \times \Delta^1\to \C$. We want to prove that the functor
\[\xymatrix{HZ\pi: \CP \times \Delta^1 \ar[r]^-{\pi\times \id} & \C\times \Delta^1 \ar[r]^-{Z\times \id} & \CP \times \Delta^1 \ar[r]^-{H} & \C}\]
is equivalent to $H$. %

We will achieve this by constructing an equivalence $H(f)\xrightarrow{\sim} HZ\pi(f)$ natural in $f\in \CP$. To make the argument more transparent, we first construct this equivalence for all $f$, and then we argue it is natural.\\

\underline{Step 1: Construction of an equivalence $H(f)\xrightarrow{\sim}HZ\pi(f)=H(0)$}

We first introduce some notation. Let $f,g:G\to P$. We let $f\tplus g$ denote the universal map $G\oplus G\to P$ induced by $f$ and $g$: we can express it as the composition
\[\xymatrix{G\oplus G \ar[r]^-{f\oplus g}& P \oplus P \ar[r]^-\nabla & P}\]
where $\nabla:P\oplus P \to P$ denotes the fold map. Functorially, the biproduct functor $\oplus:\C\times \C\to \C$ induces a functor $\oplus: \CP\times \CP\to \CPP$, %
and the fold map induces $\nabla_*:\CPP\to \CP$. 
We define the functor $\tplus$ as the composition $\nabla_*\circ \oplus: \CP\times \CP\to \CP$.

Let $\iota_0:G\to G\oplus G$ denote the coproduct inclusion into the first summand. By naturality of $H$, the following diagram commutes, %
\[\xymatrix{G \ar[r]^-{\iota_0} \ar[d]_-{Hf} & G \oplus G \ar[d]^-{H(f\tplus 0)} \\ G \ar[r]_-{\iota_0} & G\oplus G}\]
since $\iota_0\circ (f\tplus 0)\simeq f$. 
Using $\iota_1$, we similarly conclude that $H(f\tplus 0)\circ \iota_1 \simeq \iota_1 \circ H(0)$. This proves that $H(f\tplus 0)\simeq Hf\oplus H0$.

We have just used that $\oplus$ is the coproduct: we will now use that $\oplus$ is the product as well. Let $\Delta:G\to G\oplus G$ denote the diagonal map. We have the following commutative diagram, where the square on the left commutes by naturality of $H$:
\begin{equation}\label{eq-magicdiagram}
\begin{tikzcd}
	{G} && {G\oplus G} && {G} \\
	\\
	{G} && {G\oplus G} && {G.}
	\arrow["{\pi_1}", from=1-3, to=1-5]
	\arrow["{\Delta}", from=1-1, to=1-3]
	\arrow["{Hf}"', from=1-1, to=3-1]
	\arrow["{Hf\oplus H0}", "H(f\tplus 0)"', from=1-3, to=3-3]
	\arrow["{\Delta}", from=3-1, to=3-3]
	\arrow["{\pi_1}", from=3-3, to=3-5]
	\arrow["{H0}", from=1-5, to=3-5]
	\arrow["{\id}", from=1-1, to=1-5, curve={height=-18pt}]
	\arrow["{\id}"', from=3-1, to=3-5, curve={height=18pt}]
\end{tikzcd}\end{equation}
Here the left arrow is $H(f)$ because $\Delta \circ (f\tplus 0)\simeq f$. This constructs the equivalence $Hf\simeq H0$.\\

\underline{Step 2: Naturality of the equivalence}

We will now prove the equivalence just constructed is natural in $f$, by exhibiting the horizontal maps in (\ref{eq-magicdiagram}) as components of natural transformations. The naturality of the top and bottom parts will be clear from the construction, so we have to prove naturality of the left and right squares.\\

\underline{Step 2a: The left square}

Informally, naturality follows from applying $H$ to the following natural commutative triangle:
\[\xymatrix{G\ar[r]^-{\Delta}\ar[d]_-f & G\oplus G \ar[ld]^-{f\tplus 0} \\ P.}\]
More carefully phrased, we have a natural transformation $\Delta$ from the identity of $\CP$ to the functor $-\tplus 0$; the latter functor is the composition
\[\xymatrix{\CP \ar[r]^-\Delta & \CP\times \CP \ar[r]^-{\id \times Z\pi} & \CP \times \CP \ar[r]^-\oplus & \CPP \ar[r]^-{\nabla_*} &\CP.}\]
We can now compose $\Delta$ with $H$ in the sense of the following whiskering:
\[\begin{tikzcd}
	{\CP \times \Delta^1} && {\CP \times \Delta^1} & {\C}
	\arrow["{\id \times \id}"{name=0}, from=1-1, to=1-3, curve={height=-12pt}]
	\arrow["{(-\tplus 0)\times \id}"{name=1, swap}, from=1-1, to=1-3, curve={height=12pt}]
	\arrow["{H}", from=1-3, to=1-4]
	\arrow[Rightarrow, "{\Delta\times \id}", from=0, to=1, shift right=2, shorten <=2pt, shorten >=2pt]
\end{tikzcd}\] 
getting a natural transformation $H\Rightarrow H(-\tplus 0)$.\\

\underline{Step 2b: Natural equivalence between $H(-\tplus 0)$ and $H\oplus HZ\pi$}

We will now prove the naturality of the identification of $H(f\tplus 0)$ as a map that satisfies the universal property of the coproduct of the maps $\iota_0\circ H(f): G \to G \oplus G$ and $\iota_1 \circ H(0): G \to G \oplus G$.

We have the natural commutative triangle
\[\xymatrix{G\ar[r]^-{\iota_0}\ar[d]_-f & G\oplus G \ar[ld]^-{f\tplus 0} \\ P.}\]
More formally, the $\iota_0$ above is the component of a natural transformation $\iota_0:\id_\CP\Rightarrow - \tplus 0$; compose it with $H$.

We also have this natural commutative triangle:
\[\xymatrix{G\ar[r]^-{\iota_1}\ar[d]_-0 & G\oplus G \ar[ld]^-{f\tplus 0} \\ P.}\]
More formally, the $\iota_1$ above is the component of a natural transformation $\iota_1:Z\pi\Rightarrow - \tplus 0$. Now compose it with $H$.

This proves that we can regard $H(-\tplus 0)$ as a representative for $H \oplus HZ\pi$, which is the composition
\[\xymatrix@C+2pc{\C_{/P} \times \Delta^1 \ar[r]^-{\Delta} & \C_{/P} \times \Delta^1 \times \C_{/P} \times \Delta^1\ar[r]^-{H \times HZ\pi} & \C \times \C  \ar[r]^-{\oplus} & \C.}\]

\underline{Step 2c: The right square}

Making a minor abuse of notation for $\pi_1$, we have a natural transformation $\pi_1$ from $\oplus:\C\times \C \to \C$ to $\pi_1:\C\times \C\to \C$, which, when whiskered with the functor
\[\xymatrix@C+2pc{\C_{/P} \times \Delta^1 \ar[r]^-{\Delta} & \C_{/P} \times \Delta^1 \times \C_{/P} \times \Delta^1\ar[r]^-{H \times HZ\pi} & \C \times \C,}\]
gives the desired natural transformation $H\oplus HZ\pi\Rightarrow HZ\pi$.
\end{proof}
\end{prop}

We will now prove that $\Nat(\id_{\Grp(\S)},\id_{\Grp(\S)})$ is equivalent to $\End(\bS)$. First, we need a lemma.

\begin{lemma}\label{lemma-natff} Let $F:\D\to \C$ be a functor of $\infty$-categories. There is a homotopy equivalence
	\[\Nat(F,F)\simeq \Fun_F(\D, \Fun(B\N,\C))^\sim.\] Here the latter is the maximal subspace of the full subcategory of $\Fun(\D,\Fun(B\N,\C))$ on the functors that make the following diagram commute
\[\begin{tikzcd}
			\D \arrow[dr, "F"'] \arrow[rr] & & \Fun(B\N,\C) \arrow[dl, "p"] \\
			 & \C & 
		\end{tikzcd}
		\]
where $p:\Fun(B\N,\C)\to \Fun(*,\C)\simeq \C$ is induced by the inclusion $*\to B\N$.

In particular,
\[ \NatCol_P(\C)\simeq \Fun_\pi(\CP,\Fun(B\N,\C)) \hspace{.5cm} \text{and} \hspace{.5cm} \Nat(\id_\C,\id_\C)\simeq \Fun_{\id_\C}(\C,\Fun(B\N,\C)).\]
\end{lemma}

\begin{proof}
Note that a natural transformation $F \Rightarrow F$ is a functor $\alpha: \D \to \C^{\Delta^1}$
	such that \[\{0\}^*(\alpha), \{1\}^*(\alpha): \D \to \C\] are equivalent to $F$. 
	Here, $\{0\}, \{1\}: \Delta^0 \to \Delta^1$ are the inclusion maps.
	
	By the universal property of the pullback, this means $\alpha$ is equivalently a functor $\alpha:\D\to \C^{\Delta^1}\times_{\C^{\partial \Delta^1}} \C^{\Delta^0}$ making the following diagram commute:
	\begin{center}
		\begin{tikzcd}
			\D  \arrow[dr, "F"'] \arrow[rr, "\alpha"] & &  \C^{\Delta^1} \times_{\C^{\partial \Delta^1}} \C^{\Delta^0} \arrow[dl, "\pi_2"] \\
			& \C. &
		\end{tikzcd}
	\end{center}
Note that we have the following pushout diagram of $\infty$-categories:
	\begin{center}
		\begin{tikzcd}
			\partial \Delta^1 \arrow[r, hookrightarrow] \arrow[d] & \Delta^1 \arrow[d] \\
			\Delta^0 \arrow[r] & B\mathbb{N} 
		\end{tikzcd}
	\end{center}
	which implies that $\C^{\Delta^1} \times_{\C^{\partial \Delta^1}} \C^{\Delta^0}\simeq \Fun(B\N,\C)$, and the conclusion follows.
\end{proof}

Note that the fiber of $p:\Fun(B\N,\C)\to \C$ over $c\in \C$ is the space of endomorphisms $\End(c)$.

\begin{prop} \label{prop-nat-end}
The evaluation map
	\[\ev_{\bS}:\Nat(\id_{\Grp(\S)},\id_{\Grp(\S)}) \to \End(\bS)\]
is an equivalence of spaces. Here $\End(\bS)=\Map_{\Grp(\S)}(\bS,\bS)$.
\end{prop}
\begin{proof} 
Note that 
\[\Fun_\id(\Grp(\S),\Fun(B\N,\Grp(\S)))^\sim = \Fun^L_\id(\Grp(\S), \Fun(B\N,\Grp(\S)))^\sim,\] 
since as $p$ creates colimits, every functor in the left hand side preserves colimits.
Combining this with \cref{lemma-natff}, we obtain an equivalence 
\[\Nat(\id,\id) \simeq \Fun^L_\id(\Grp(\S), \Fun(B\N,\Grp(\S)))^\sim.\]
	
	Now, we have a commutative diagram 
	\begin{center}
		\begin{tikzcd}
			\Fun^L(\Grp(\S),\Fun({B\mathbb{N}},\Grp(\S))) \arrow[r, "\sim", "\ev_\bS"'] \arrow[d, "p_*"] & \Fun({B\mathbb{N}},\Grp(\S)) \arrow[d, "p"] \\
			\Fun^L(\Grp(\S),\Grp(\S)) \arrow[r, "\sim", "\ev_\bS"'] &  \Grp(\S).
		\end{tikzcd}
	\end{center}
By \cref{prop-freely-generated}, the horizontal maps are equivalences: note that $\Fun(B\N, \Grp(\S))$ is semiadditive since $\Grp(\S)$ is \cite[Example 2.2]{ggn}. We consider $\id_{\Grp(\S)}$ in the bottom left corner, which maps to $\bS\in \Grp(\S)$ by $\ev_\bS$. Taking fibers vertically and taking maximal subspaces, we conclude that 
	$$\ev_\bS: \Nat(\id,\id) \to \End(\bS)$$
	is an equivalence.
\end{proof}

\begin{cor} \label{cor-natcol-omega}Let $P$ be an $E_\infty$-group. There is a homotopy equivalence
\[\NatCol_P(\Grp(\S))\simeq \Omega^\infty \bS.\]
In particular,
\[\pi_0(\NatCol_P(\Grp(\S))) \cong \mathbb{Z},\]
and thus every natural collection of endomorphisms over $P$ is equivalent to $\bP_n$ for some $n\in \Z$.
\begin{proof}
By \cref{prop-nat-natcol} and \cref{prop-nat-end}, we have
\[\NatCol_P(\Grp(\S))\simeq \End(\bS).\]
Now note that
\[\End(\bS)=\Map_{\Grp(\S)}(\bS,\bS)\simeq \Map_{\Grp(\S)}(\Sigma^\infty_+(*),\bS) \simeq \Map_\S(*,\Omega^\infty \bS) \simeq \Omega^\infty \bS.\]

By \cref{rmk-pn}, the $\bP_n$, $n\in \Z$ are pairwise non-equivalent natural collections of endomorphisms over $P$, so this proves the last statement.
\end{proof}
\end{cor}

 Let us now use the above to prove that the Thom isomorphism is generally not given by $S^0$-base change.
 
 \begin{defn}
  We say a map of $E_\infty$-groups $f: G \to P$ is {\it torsion} if there exists a  natural number $n>0$ such that 
  $f \circ \mathbb{P}_n: G \to P$ is trivial. Otherwise, we say it is {\it non-torsion}.
 \end{defn}

\begin{rmk}
Let $f:G\to H$ be a map of $E_\infty$-groups. If there exists a $k\geq 0$ such that 
   the map of groups 
   $$\pi_k(f) : \pi_k(G) \to \pi_k(H)$$
   is non-torsion, then $f$ is non-torsion.
\end{rmk}

 \begin{ex}\begin{enumerate}
   \item A map of groups $f: G \to H$ is non-torsion if and only if it is non-torsion in the classical group-theoretical sense.
   \item By the previous remark, if a map of groups $f:G \to H$ is non-torsion then the induced map 
   $$Bf: BG \to BH$$
   is non-torsion as well.
  \end{enumerate}
 \end{ex}

\begin{prop} \label{prop-tor}
 Let %
 $f: G \to P$ be a map of $E_\infty$-groups that is non-torsion.
 Let $H$ be a natural collection of endomorphisms over $P$ such that the following diagram commutes, where $R$ denotes the unit of $P$.
 \[\xymatrix@R-1.5pc{G \ar[rd]^-{\{R\}} \ar[dd]_-{H(f)} \\ & P \\ G \ar[ru]_-f }\]
 Then $H(f) \simeq \{e\}: G \to G$, where $e$ denotes the unit of $G$.
\end{prop}

\begin{proof}
 By \cref{cor-natcol-omega}, $H(f) \simeq \mathbb{P}_n$, which implies that $f \circ \mathbb{P}_n \simeq f \circ H(f) \simeq \{ R \}$, which 
 is only possible if $n = 0$ since $f$ is non-torsion, giving us the desired result.
\end{proof}

We will now apply \cref{prop-tor} to the Thom isomorphism theorem. We first need a definition.

\begin{defn} \label{def-solid}Let $R$ be an $E_\infty$-ring spectrum. An $E_\infty$-$R$-algebra $A$ is \emph{solid} if the multiplication map $A \sma_R A \to A$ is an equivalence.
\end{defn}

Since multiplying by the unit is equivalent to the identity, by the 2-out-of-3 property for equivalences we conclude that $A$ is solid if and only if any coproduct inclusion $A\to A\sma_R A$ is an equivalence.

\begin{thm} \label{thm-thomisnots0}
 Let $R$ be an $E_\infty$-ring spectrum. %
 Let $H$ %
 be a natural collection of endomorphisms over $\Pic(R)$. Suppose that for all $f:G\to \Pic(R)$, we have
 that $f \circ H(f) \simeq \{ R \}$, and additionally the induced map $M(H(f)): R[G] \to M(f)$ satisfies $S^0$-base change.
 Then, for any non-torsion $f: G \to \Pic(R)$, $Mf$ is a solid $E_\infty$-$R$-algebra.
\end{thm}

\begin{proof}
 
 Let $f: G \to \Pic(R)$ be a non-torsion map. By the previous proposition, the map $H(f): G \to G$ is trivial. 

 By the results of \cref{section-thom}, we have equivalences
 $$Mf \sma_R Mf \simeq M(G \times G \xrightarrow{ f \times f} \Pic(R) \times \Pic(R) \xrightarrow{ \ \mu \ } \Pic(R)) \simeq 
 M(G \times G \xrightarrow{ \ \mu \ } G \xrightarrow{ \ f \ } \Pic(R)),$$
 where the last step follows from $f$ being a map of $E_\infty$-monoids. 
 Similarly, we have equivalences:
 \begin{align*}
&\hspace{2cm} R[G] \sma_R Mf \simeq M(G \xrightarrow{ \ \{ R \} \ } \Pic(R)) \sma_R M(G \xrightarrow{ \ f \ } \Pic(R)) \simeq \\
&\simeq M(G \times G \xrightarrow{ \ \{R\} \times f \ } \Pic(R) \times \Pic(R) \xrightarrow{ \ \mu \ } \Pic(R)) \simeq M(G \times G \xrightarrow{ \ \pi_2 \ } G \xrightarrow{ \ f \ } \Pic(R)).\end{align*}
Since $H(f) \simeq \{ e \}$ and $M(H(f))$ satisfies $S^0$-base change, we get the following commutative diagram: 
 \begin{center}
  \begin{tikzcd}[row sep=0.3in, column sep=0.5in]
    R[G] \sma_R Mf \arrow[d, dash, "\simeq"] \arrow[r, "M(\{e\}) \sma_R \id","\sim"']& Mf \sma_R Mf \arrow[d, dash, "\simeq"] \\
    M(G \times G \xrightarrow{ \ \pi_2 \ } G \xrightarrow{ \ f \ } \Pic(R)) \arrow[r, "M(\{e\} \times \id)"] & 
    M(G \times G \xrightarrow{ \ \mu \ } G \xrightarrow{ \ f \ } \Pic (R)). 
  \end{tikzcd}
 \end{center}
The map $\{e\} \times \id: G \times G \to G \times G$ factors as $\{e\} \times \id = \iota_2\pi_2$.
Thus, we get the following commutative diagram of $E_\infty$-monoids over $\Pic(R)$.

 \begin{center}
  \begin{tikzcd}[row sep=0.5in, column sep=0.5in]
  G \arrow[rr, "\iota_2"] \arrow[drrr, "\id_G"'] \arrow[rrrr, "\id_G", bend left = 20]& &G \times G \arrow[dr, "\pi_2"'] \arrow[rr, "\pi_2"] \arrow[rrrr, "\{e\} \times \id_G"', near start, bend right = 10] & &G \arrow[dl, "\id_G"] \arrow[rr, "\iota_2"] & &G \times G \arrow[dlll, "\mu"] \\
   & & &G \arrow[d, "f"] & & & \\
   & & & \Pic(R) & & &
  \end{tikzcd}
 \end{center}
 Applying the Thom construction to this diagram, we get
  \begin{center}
   \begin{tikzcd}[row sep=0.5in, column sep=0.9in]
    Mf \arrow[r, "M(\iota_2)"] \arrow[rr, "\id_{Mf}", bend left = 20]& R[G] \sma_R Mf \arrow[r, "M(\pi_2)"] \arrow[rr, "\simeq", "M(\{e\}\times \id)"', bend right = 20] & Mf \arrow[r, "M(\iota_2)"] & Mf \sma_R Mf
   \end{tikzcd}
 \end{center}
 As $M(\{e\} \times \id)$ is an equivalence, there exists an equivalence $K: Mf \sma_R Mf \to  R[G] \sma_R Mf$ such that $M(\{e\} \times \id) \circ K \simeq \id_{Mf \sma_R Mf}$, which gives us the following commutative diagram: 
\begin{center}
   \begin{tikzcd}[row sep=0.5in, column sep=0.9in]
    Mf \arrow[r, "M(\iota_2)"] \arrow[rr, "\id_{Mf}", bend left = 20]& Mf \sma_R Mf \arrow[r, "M(\pi_2)K"] \arrow[rr, "\id_{Mf \sma_R Mf}"', bend right = 20] & Mf \arrow[r, "M(\iota_2)"] & Mf \sma_R Mf
   \end{tikzcd}
 \end{center}
whence we conclude that $M(\iota_2): Mf \to Mf \sma_R Mf$ is an equivalence. 
\end{proof}

\begin{ex} \label{ex-picmup}
 Recall that $MUP = M(BU \times \mathbb{Z} \to \Pic(\mathbb{S}))$. The map of $E_\infty$-groups $BU \times \mathbb{Z} \to \Pic(\mathbb{S})$ 
 is non-torsion as the induced map $\mathbb{Z} = \pi_0(BU \times \mathbb{Z}) \to \pi_0(\Pic(\mathbb{S})) = \mathbb{Z}$ is just 
 the identity map. 
 \par 
 In order to prove that the Thom isomorphism $MUP \sma MUP \simeq MUP \sma \bS[BU\times \Z]$ is not given by $S^0$-base change for a map $M(H(f))$ as in \cref{thm-thomisnots0}, 
 it suffices to prove that $MUP$ is not a solid $E_\infty$-ring spectrum. If $MUP$ were solid, then its rational homology $H_*(MUP,\mathbb{Q})$ would be a solid $\Q$-algebra, as 
 rational homology satisfies the K{\"u}nneth formula. This means that the inclusion map of $\mathbb{Q}$-algebras
 $$H_*(MUP,\mathbb{Q}) \xrightarrow{ \ \iota_1 \ } H_*(MUP,\mathbb{Q}) \otimes_{\mathbb{Q}} H_*(MUP,\mathbb{Q})$$
 is an isomorphism, which is only possible if the dimension of $H_*(MUP,\mathbb{Q})$  is 0 or 1. Indeed, if $\beta = \{ x_i \}$ is a basis for $H_*(MUP,\mathbb{Q})$, 
 then $\beta \otimes \beta = \{x_i \otimes x_j \}_{i,j}$ is a basis for $H_*(MUP,\mathbb{Q}) \otimes_{\mathbb{Q}} H_*(MUP,\mathbb{Q})$ 
 and an isomorphism necessitates a bijection $\beta \xrightarrow{ \ \cong \ } \beta \otimes \beta$, $x_i \mapsto x_i \otimes 1$, 
 which implies $|\beta|=0,1$. However, this clearly does not hold as $H_*(MUP,\mathbb{Q}) = \pi_*(MUP) \otimes \mathbb{Q}$ is a polynomial algebra on infinitely many generators.
\end{ex}

\section{Base change for maps of \texorpdfstring{$E_\infty$}{E\textunderscore infty}-ring spectra} \label{section-base-rings}
Following up on \cref{section-basechange}, where we introduced the notion of $X$-base change, we now specialize to the case where $\C=\CAlg(\Sp)$ (similarly, we could take $\CAlg(\Mod_R)$ for $R$ an $E_\infty$-ring spectrum), in which case for $f:A\to B$ the map (\ref{eq-xbase}) becomes
\[(X \otimes A) \sma_A B \to X\otimes B.\]
We will now consider conditions on $f$ which guarantee that the map displayed above is an equivalence, i.e. that $f$ satisfies $X$-base change, for some class of pointed spaces $X$. Note, for example, that it is unlikely to hold for non-connected $X$: a map $A\to B$ satisfies $S^0$-base change if and only if $A\sma B \to B\sma B$ is an equivalence (a concrete counterexample is given in \cref{ex-s0}). We first look at the case where $f$ is an étale map. Then we look at the case where $B$ is a solid $E_\infty$-$A$-algebra.
As a particular case, we look at inversion of a homotopy element. Finally, we consider Galois extensions.

\subsection{Étale maps}

Following \cite[7.5.0.1/2]{ha}, a map of (ordinary) commutative rings $A\to B$ is \emph{étale} if $B$ is finitely presented as an $A$-algebra, $B$ is flat as an $A$-module, and there exists an idempotent element $e\in B\otimes_A B$ such that the multiplication map $B\otimes_A B\to B$ induces an isomorphism $(B\otimes_A B)[e^{-1}] \cong B$.
For the following definition, we follow the terminology of \cite[7.5.0.4]{ha} for étale maps, and that of \cite{mccarthy-minasian} for $THH$-étale and $TAQ$-étale maps. Note that in \cite{rognes-galois}, the latter two are called ``formally symmetrically étale'' and ``formally étale'' maps, respectively.

\begin{defn} Let $f:A\to B$ be a map of $E_\infty$-ring spectra. We say it is:
\begin{enumerate}
\item \emph{étale} if $\pi_0(A)\to \pi_0(B)$ is étale and $B$ is flat as an $A$-module, i.e. the natural map \[\pi_*(A)\otimes_{\pi_0(A)} \pi_0(B) \to \pi_*(B)\] is an isomorphism,
\item \emph{$THH$-étale} if the natural map $B\to THH^A(B)$ is an equivalence,
\item \emph{$TAQ$-étale} if the $E_\infty$-cotangent complex $L_{B/A}$ is contractible.
\end{enumerate}
\end{defn}
The natural map $A\to THH^A(B)=S^1\otimes_A B$ (where $\otimes_A$ denotes the tensor of $\CAlg_A$ over spaces) comes from the inclusion of a basepoint into $S^1$. As noted in \cite[Section 5]{mathew-thh}, this map is an equivalence if and only if $A\to B$ is a \emph{0-cotruncated} map in $\CAlg(\Sp)$, meaning that the map of spaces $f^*:\Map_{\CAlg(\Sp)}(B,C)\to\Map_{\CAlg(\Sp)}(A,C)$ is a covering space for all $C\in \CAlg(\Sp)$, i.e. it has discrete homotopy fibers over any basepoint. This is easily proven using the tensor -- internal hom adjunction of $\CAlg(\Sp)$.
The cotangent complex in the $\infty$-categorical context is defined in \cite[7.3]{ha}.

Any étale map is $THH$-étale \cite[7.5.4.6]{ha} and any $THH$-étale map is $TAQ$-étale \cite[9.4.4]{rognes-galois}. Note that if $A$ and $B$ are connective and satisfy some finiteness hypothesis (such as $B$ being finitely presented as an $A$-algebra), then a $TAQ$-étale map is étale \cite[8.9]{dagvii} and thus in this case all three étaleness conditions are equivalent.

Mathew has proven that if $A\to B$ is étale, then it satisfies $S^1$-base change \cite[1.3]{mathew-thh}. In light of \cref{thm-s1enough}, we get:

\begin{cor} \label{thm-etale-change} Let $X$ be a connected pointed space. Let $f:A\to B$ be an étale map of $E_\infty$-ring spectra. Then $f$ satisfies $X$-base change.
\end{cor}

Note that in \cite[5.2]{mathew-thh} Mathew also proved that if $A\to B$ is $THH$-étale, then it satisfies $X$-base change for all simply connected pointed spaces $X$. He stated it for faithful Galois extensions (see \cref{section-galois}), but his proof only uses $THH$-étaleness, which is satisfied for Galois extensions \cite[9.2.6]{rognes-galois}.

\subsection{Solidity and inversion of a homotopy element}
Let $R$ be an $E_\infty$-ring spectrum and $x\in \pi_0(R)$. Lurie proves \cite[7.5.0.6/7]{ha} that there exists an $E_\infty$-ring spectrum $R[x^{-1}]$ with an étale map of $E_\infty$-ring spectra $R\to R[x^{-1}]$ which realizes the algebraic localization morphism $\pi_0(R)\to \pi_0(R)[x^{-1}]$.

There are interesting instances where we want to invert a homotopy element $x$ that is not in degree $0$, as we shall see below. In \cite[4.3.17]{lurie-chromatic2}, Lurie constructs $R[x^{-1}]$ for $x \in \pi_n(R)$, $n\in \Z$, and he gives the following universal property for the $E_\infty$-$R$-algebra $R[x^{-1}]$. Consider $x$ as a map $\Sigma^nR\to R$ in $\Mod_R$. For any $A\in \CAlg_R$, 
\[\Map_{\CAlg_R}(R[x^{-1}],A) \simeq \begin{cases}  \text{contractible} & \text{if } x \text{ induces an equivalence } A \sma_R \Sigma^n R \stackrel{\sim}{\to} A \sma_R R, \\ \emptyset & \text{otherwise.}
\end{cases}\]

Note that if $A\in \CAlg_R$, then the unit $R\to A$ allows us to consider $x$ as an element of $\pi_*(A)$, and thus to consider $A[x^{-1}]$.

\begin{lemma} \label{lemma-loc-smashing} Let $R$ be an $E_\infty$-ring spectrum, $x\in \pi_*(R)$ and $A\in \CAlg_R$. Then the canonical pushout map
\[A\sma_R R[x^{-1}] \to A[x^{-1}]\]
is an equivalence of $E_\infty$-$A$-algebras.
\begin{proof} Let $B\in \CAlg_A$. Note that $A\sma_R R[x^{-1}]$ is in $\CAlg_A$ as the extension of scalars via $R\to A$ of $R[x^{-1}]$, so $\Map_{\CAlg_A}(A\sma_R R[x^{-1}],B)$ is contractible if $B\sma_R \Sigma^n R\to B\sma_R R$ is an equivalence, and empty otherwise. The mapping space $\Map_{\CAlg_A}(A[x^{-1}],B)$ is contractible or empty under the same conditions, so the result follows.
\end{proof}
\end{lemma}

Unfortunately, when $x$ is not in degree $0$ the map $R\to R[x^{-1}]$ may not be étale: indeed, it may not be flat. For example, if $R$ is connective then any flat $R$-module is necessarily connective, as follows from the definition \cite[7.2.2.11]{ha}.
To remedy this, we will use the notion of solidity which we introduced in \cref{def-solid}.

\begin{ex} \label{ex-invsolid} By \cref{lemma-loc-smashing}, if $R$ is an $E_\infty$-ring spectrum and $x\in \pi_*(R)$, then $R[x^{-1}]$ is a solid $E_\infty$-$R$-algebra, since $R[x^{-1}][x^{-1}]\simeq R[x^{-1}]$.
\end{ex}

\begin{thm} \label{prop-solid} Let $R$ be an $E_\infty$-ring spectrum and $A$ be a solid $E_\infty$-$R$-algebra. Then the unit $R\to A$ satisfies $X$-base change for any connected pointed space $X$.
\begin{proof}
By \cref{thm-s1enough}, it suffices to prove that $f$ satisfies $S^1$-base change, namely that the pushout map
\begin{equation} \label{eq-pushs1} (S^1\otimes R)\sma_R A\to S^1\otimes A\end{equation}
is an equivalence. Write $S^1$ as the pushout of $*\leftarrow S^0\to *$. Since $-\otimes R$ preserves pushouts, we recover the classical formula for topological Hochschild homology $S^1\otimes R\simeq R \sma_{R\sma R} R$. Using this equivalence for $R$ and for $A$ in (\ref{eq-pushs1}), we see that we want to prove that the pushout map
\[(R\sma_{R\sma R} R) \sma_R A \to A\sma_{A\sma A} A\]
is an equivalence. The left hand side simplifies to $R\sma_{R\sma R} A$. 

Consider the following commutative diagram:
\[\xymatrix{
R & R \ar[r] \ar[l] &R \\
R \ar[u] \ar[d] & \bS \ar[l] \ar[u]\ar[d] \ar[r] & R \ar[u] \ar[d] \\
A & A \ar[l] \ar[r] & A.
}\]
Comparing pushouts computed horizontally or vertically, we get $R\sma_{R\sma R} A\simeq A\sma_{R\sma A} A$. We are using \cite[5.5.2.3]{htt}.

Similarly, consider the following commutative diagram:
\[\xymatrix{
A & A \ar[r] \ar[l] & A \\
R \ar[u] \ar[d] & \bS \ar[r]\ar[l]\ar[d]\ar[u] & A \ar[u] \ar[d] \\
A & A \ar[l]\ar[r] & A.
}\]
Comparing pushouts computed horizontally or vertically, we get \[A \sma_{R\sma A} A\simeq (A\sma_R A) \sma_{A\sma A} A\simeq A \sma_{A\sma A} A\]
where the second equivalence comes from the solidity of $A$. This finishes the proof.
\end{proof}
\end{thm}

As a direct corollary of \cref{prop-solid}, \cref{ex-invsolid} and \cref{lemma-loc-smashing}, we obtain:
\begin{cor} \label{cor-tensor-localization} Let $X$ be a connected pointed space. Let $R$ be an $E_\infty$-ring spectrum and $x\in \pi_*(R)$. Then $R\to R[x^{-1}]$ satisfies $X$-base change, i.e. the canonical map
\[(X\otimes R) \sma_R R[x^{-1}]\to X\otimes R[x^{-1}]\]
is an equivalence, so we get an equivalence
\[(X\otimes R)[x^{-1}] \simeq X\otimes R[x^{-1}].\]
\end{cor}

\begin{rmk} \label{rmk-loc-smash} Let $R$ and $T$ be $E_\infty$-ring spectra. Let $x\in \pi_*(R)$. Applying \cref{lemma-loc-smashing} to $A=R\sma T$ where $R\sma T$ is an $E_\infty$-$R$-algebra via the canonical map $R\to R\sma T$, we get an equivalence of $E_\infty$-$R$-algebras %
 $R[x^{-1}]\sma T \simeq (R\sma T)[x^{-1}]$.
\end{rmk}

Putting \cref{cor-tensor-localization}, \cref{thm-tensorthom}
and \cref{rmk-loc-smash} together, we deduce:

\begin{cor} Let $G$ be an $E_\infty$-group, $R$ be an $E_\infty$-ring spectrum and $f:G\to \Pic(R)$ be an $E_\infty$-map. Let $x\in \pi_*(Mf)$. Let $X$ be a connected pointed space. Then
\[X\otimes_R ((Mf)[x^{-1}]) \simeq (Mf)[x^{-1}] \sma \bS[X\odot G]\]
as $E_\infty$-$R$-algebras.
\end{cor}

Taking $R=\bS$ and the Thom spectrum to be trivial, we get:

\begin{cor} \label{thm-sgx} Let $G$ be an $E_\infty$-group and $x\in \pi_*(\bS[G])$. Let $X$ be a connected pointed space. Then
\[X\otimes (\bS[G][x^{-1}]) \simeq \bS[G][x^{-1}] \sma \bS[X\odot G]\]
as $E_\infty$-ring spectra.
\end{cor}

\begin{ex} \label{ex-s0} Let $KU$ denote the periodic complex topological $K$-theory $E_\infty$-ring spectrum. Snaith \cite{snaith79}, \cite{snaith81} proved that $KU\simeq \bS[K(\Z,2)][x^{-1}]$ as homotopy commutative ring spectra (i.e. commutative monoids in the homotopy category of spectra), where $x\in \pi_2 \bS[K(\Z,2)]$ is induced by the fundamental class in $K(\Z,2)$. See \cite[6.5.1]{lurie-chromatic2} for one improvement of such an equivalence to an equivalence of $E_\infty$-ring spectra. \cref{thm-sgx} gives an equivalence of $E_\infty$-ring spectra
\[X\otimes KU\simeq KU \sma \bS[X\odot K(\Z,2)]\]
for any connected pointed space $X$. 
In \cite{stonek-thhku}, the second-named author worked in a model-categorical setting and got that description of $X\otimes KU$ when $X$ is an $n$-sphere or an $n$-torus, $n\geq 1$ (an inductive proof similar to the one of \cref{thm-s1enough} would have allowed the author to obtain the above formula for $X\otimes KU$ in that same setting). A different description as a free $E_\infty$-$KU$-algebra was also given; that description involves properties special to $KU$ which do not generalize to other Thom spectra with a homotopy element inverted, so they are not recovered here.

Note that the formula above proves that $-\otimes KU$ is a stable invariant, in the sense of Remarks \ref{rmk-odotstable} and \ref{rmk-otimesstable}. This was observed in \cite[5.4]{lindenstrauss-richter}.

To conclude this example, let us remark that the map $\bS[K(\Z,2)]\to KU$ which inverts $x$ does not satisfy $S^0$-base change, even though it satisfies $X$-base change for connected $X$ by \cref{cor-tensor-localization}. Indeed, the spectra $S^0\otimes KU\simeq KU \sma KU$ and $KU \sma K(\Z,2)_+$ have different homotopy groups, see e.g. \cite[17.34, 16.30]{switzer}. This is the example used in \cref{rmk-notclosed} to prove that base change is not closed under %
pullbacks of spaces. 
\end{ex}

\begin{ex} In \cite[1.2]{westerland}, Westerland gives a higher analog of Snaith's theorem. First, recall that if $p$ is a prime and $E_n$ denotes the $n$-th Morava $E$-theory spectrum at the prime $p$, then $E_1\simeq KU_p$, the $p$-completed periodic complex $K$-theory spectrum. The extended Morava stabilizer group $\G_n$ acts on $E_n$, and $E_n^{h\G_n}\simeq L_{K(n)}\bS$, the $K(n)$-local sphere; here $K(n)$ denotes the $n$-th Morava $K$-theory at the prime $p$. The extended Morava stabilizer group has a subgroup denoted $S\G_n^{\pm}$ by Westerland: he establishes an equivalence of $E_\infty$-ring spectra
\[E_n^{hS\G_n^{\pm}} \simeq (L_{K(n)} \bS[K(\Z_p,n+1)])[\rho_n^{-1}]\]
where $\rho_n$ is a higher analogue of the Bott element and $p$ is odd.
\cref{cor-tensor-localization} gives a first modest step towards the calculation of $X\otimes E_n^{hS\G_n^{\pm}}$ where $X$ is a connected pointed space, i.e. its calculation reduces to the determination of $X \otimes L_{K(n)} \bS[K(\Z_p,n+1)]$.
\end{ex}

\begin{ex} \label{ex-mup2} Let $MUP$ denote the periodic complex bordism $E_\infty$-ring spectrum. In \cite{snaith81}, Snaith proved that $MUP\simeq \bS[BU][x^{-1}]$ as homotopy commutative ring spectra, where $x\in \pi_2\bS[BU]$. %
This suggests a computation of $X\otimes MUP$ for $X$ a connected pointed space using \cref{thm-sgx}, but it turns out that one cannot proceed as straightforwardly as in \cref{ex-s0} because this equivalence of homotopy commutative ring spectra \emph{cannot} be lifted to an equivalence of $E_\infty$-ring spectra, though it can be lifted to an equivalence of $E_2$-ring spectra \cite{hahn-yuan}. Of course, \cref{thm-sgx} does give an equivalence of $E_\infty$-ring spectra \[X\otimes \bS[BU][x^{-1}] \simeq \bS[BU][x^{-1}] \sma \bS[X\odot BU].\]
Since $THH$ is actually an invariant of $E_1$-ring spectra and $MUP$ and $\bS[BU][x^{-1}]$ are equivalent as $E_1$-ring spectra, then as $BBU\simeq SU$ we get an equivalence of spectra
\[THH(MUP) \simeq MUP \sma \Sigma^\infty_+(SU).\]
In \cref{ex-mup} we proved that $THH(MUP)\simeq MUP \sma \bS[U]$ as $E_\infty$-ring spectra. Putting these two results together, we conclude that there is an equivalence of spectra
\[MUP \sma SU_+ \simeq MUP \sma U_+.\]
In other words, the (unreduced) $MUP$-homology groups of $SU$ and of $U$ are abstractly isomorphic.

In fact, one can compute these groups. To compute $MU_*(U)$, one can use the Atiyah-Hirzebruch spectral sequence $E^2_{*,*}=\widetilde{H}_*(U,\pi_*(MU))\Rightarrow \widetilde{MU}_*(U)$. Since $\pi_*(MU)$ is free over $\Z$, the $E_2$-page is $\widetilde{H}_*(U)\otimes_\Z \pi_*(MU) %
\cong \widetilde{E}(y_1,y_3,\dots) \otimes_\Z P(x_2,x_4,\dots)$ where $|y_i|=i$ and $|x_j|=j$; here $\widetilde{E}$ denotes a non-unital exterior algebra over $\Z$ and $P$ denotes a polynomial algebra over $\Z$.
Thus, the $E_2$-page has a checkerboard pattern and the spectral sequence collapses at that page. We get that $MU_*(U)\cong \widetilde{MU}_*(U)\oplus MU_*\cong E(y_1,y_3,\dots) \otimes_\Z P(x_2,x_4,\dots)$ where $E$ denotes a (unital) exterior algebra, and therefore,
\[MUP_*(U)\cong E(y_1,y_3,\dots) \otimes_\Z P(x_2,x_4,\dots)[x_2^{-1}].\]
The description of $MUP_*(SU)$ is the same but without the $y_1$ generator. By inspection, both $MUP_*(U)$ and $MUP_*(SU)$ have a direct sum of countably many copies of $\Z$ in each degree, 
so they are indeed abstractly isomorphic.
\end{ex}

\subsection{Galois extensions of \texorpdfstring{$E_{\infty}$}{E\textunderscore infty}-ring spectra} \label{section-galois}

Rognes \cite{rognes-galois} generalized the theory of Galois extensions of commutative rings to the framework of $E_{\infty}$-ring spectra. This notion has numerous applications, with a considerable source of examples coming from chromatic homotopy theory. In what follows, we will establish base change for a large class of Galois extensions. %

Let $G$ be a topological group. Recall from the discussion before \cref{prop-bg} that the $\infty$-category of objects of $\C$ with $G$-action is given by $\Fun(BG,\C)$. Whenever $X$ is a space and $A$ is an $E_\infty$-ring spectrum with $G$-action, then $X\otimes A$ is also an $E_\infty$-ring spectrum with $G$-action.

\begin{defn}
	 Let $f:A\to B$ be a map of $E_{\infty}$-ring spectra and $G$ be a finite group. Suppose, in addition, that $B$ has a $G$-action given by a functor $B:BG\to\CAlg_A$. The map $f$ is a \emph{$G$-Galois extension} if
	 \begin{enumerate}
	 	\item The canonical map $A\to B^{hG}$ is an equivalence.
	 	\item The map of $E_\infty$-$A$-algebras $B\wedge_A B\to \prod_G B$ induced by the morphisms $\{B\wedge_A B\xrightarrow{\id\wedge g}B\wedge_AB\xrightarrow{\mu} B\}_{g\in G}$ is an equivalence.
	 \end{enumerate}
	 The map $f$ is said to be a \emph{faithful $G$-Galois extension} if the functor $-\wedge_AB: \Mod_A\to \Mod_B$ is conservative, i.e. reflects equivalences.
\end{defn}
The first condition is analogous to taking fixed fields in classical Galois theory, while the second corresponds to the requirement of the extension being unramified.

Mathew proved that when $A\to B$ is a faithful Galois extension, it satisfies descent:

\begin{prop}\cite[Corollary 4.2 and (6)]{mathew-thh}\label{prop-faithful_descent}
	Let $A\to B$ be a faithful $G$-Galois extension. There is an adjoint equivalence of symmetric monoidal $\infty$-categories
		\[ 	\xymatrix@C+=1.5cm{ \Mod_B(\Fun(BG,\Sp))\ar@<.5ex>[r]^-{(-)^{hG}} & \Mod_A
\ar@<.5ex>[l]^-{-\wedge_AB}}.\]
Moreover, it yields an equivalence of $\infty$-categories
\[\Fun(BG,\CAlg(\Sp))_{B/} \simeq\CAlg(\Mod_B(\Fun(BG,\Sp))) \simeq \CAlg_A.\] 
\end{prop}
\begin{rmk}
	The construction $-\wedge_A B$ gives, a priori, a functor from $\Mod_A$ (resp. $\CAlg_A$) to $\Mod_B$ (resp. $\CAlg_B$). Since $-\wedge_A B$ is functorial in $B$, we can thus see this as a functor to the corresponding category of $G$-equivariant objects.
\end{rmk}

Next, we remark that in the context of faithful $G$-Galois extensions, the $X$-base change property and the compatibility of homotopy fixed points with tensoring are equivalent statements. This is an adaptation of \cite[4.3]{mathew-thh} which covers the case $X=S^1$; the same proof gives the following:

\begin{prop} \label{prop-x-galois}
	Let $X$ be a pointed space and $A\to B$ a faithful $G$-Galois extension. The following are equivalent:
	\begin{enumerate}
		\item $A\to B$ satisfies $X$-base change, i.e. the comparison map $(X\otimes A)\wedge_AB\to X\otimes B$ is an equivalence of $E_\infty$-$B$-algebras,
		\item $X\otimes A\to X\otimes B$ is a faithful $G$-Galois extension,
		\item The map $X\otimes A \simeq ((X\otimes A) \sma_A B)^{hG}\to (X\otimes B)^{hG}$ is an equivalence of $E_\infty$-$A$-algebras.
	\end{enumerate}
\end{prop}

 It is known that being a faithful $G$-Galois extension it is not sufficient to guarantee $X$-base change; see \cite[Section 5]{mathew-thh} for an explicit example where $X=S^1$. Nonetheless, Mathew shows in his Theorem 4.5 that $S^1$-base change does hold for faithful $G$-Galois extensions satisfying an additional condition on rational algebraic $K$-theory. In light of \cref{thm-s1enough}, we can extend his result to get $X$-base change for all connected pointed spaces $X$. In the following, $L^f_n$ denotes Miller's finite $L_n$-localization \cite{miller}.

\begin{cor}\label{cor-galois-change}
	Let $f:A\to B$ be a faithful $G$-Galois extension such that $K_0(f)\otimes\Q$ is surjective. Suppose that the $p$-localization $A_{(p)}$ is $L^f_n$-local for some $n$ (depending on $p$), for all primes $p$. Then $f$ satisfies $X$-base change for all connected pointed spaces $X$. %
\end{cor}
In particular, the other two equivalent conditions of \Cref{prop-x-galois} are satisfied as well.
\begin{rmk}
	The surjectivity condition on $K$-theory is essential as it guarantees that the comparison map on homotopy fixed points is an $L^f_n$-local equivalence \cite[Theorem 5.6]{clausen}. The additional assumption on $A$ is there only so that $A\to B$ is already a map of $L^f_n$-local spectra.
\end{rmk}
\begin{ex} As pointed out in \cite[4.6]{mathew-thh}, many Galois extensions satisfy the conditions of \cref{cor-galois-change}, including:
	\begin{enumerate}
		\item The $C_2$-Galois extension given by the complexification $KO\to KU$,
		\item The $C_{p-1}$-Galois extension given by the $p$-adic completion of the inclusion of the $p$-local Adams summand $L\to KU_{(p)}$, %
		\item The $G$-Galois extension $E_n^{hG}\to E_n$, where $E_n$ is $n$-th Morava $E$-theory and $G$ is a finite subgroup of the extended Morava stabilizer group,
		\item Any Galois extension of the different flavors of topological modular forms $TMF[1/n]$, $Tmf_0(n)$, or $Tmf_1(n)$.
	\end{enumerate}
\end{ex}

\bibliographystyle{alpha}
\bibliography{main.bib}
\end{document}